\documentclass[12pt]{amsart}
\usepackage{amssymb,float}
\usepackage[colorlinks=true]{hyperref}
\usepackage[all]{xy}
\usepackage[toc,page,title,titletoc,header]{appendix}
\usepackage{xcolor}
\usepackage{tikz-cd}
\usepackage{graphicx}

\begin{document}
\pdfoutput=1
\theoremstyle{plain}
\newtheorem{thm}{Theorem}[section]
\newtheorem*{thm1}{Theorem 1}
\newtheorem*{thm1.1}{Theorem 1.1}
\newtheorem*{thmM}{Main Theorem}
\newtheorem*{thmA}{Theorem A}
\newtheorem*{thm2}{Theorem 2}
\newtheorem{lemma}[thm]{Lemma}
\newtheorem{lem}[thm]{Lemma}
\newtheorem{cor}[thm]{Corollary}
\newtheorem{pro}[thm]{Proposition}
\newtheorem{propose}[thm]{Proposition}
\newtheorem{variant}[thm]{Variant}
\theoremstyle{definition}
\newtheorem{notations}[thm]{Notations}
\newtheorem{rem}[thm]{Remark}
\newtheorem{rmk}[thm]{Remark}
\newtheorem{rmks}[thm]{Remarks}
\newtheorem{defi}[thm]{Definition}
\newtheorem{exe}[thm]{Example}
\newtheorem{claim}[thm]{Claim}
\newtheorem{ass}[thm]{Assumption}
\newtheorem{prodefi}[thm]{Proposition-Definition}
\newtheorem{que}[thm]{Question}
\newtheorem{con}[thm]{Conjecture}
\newtheorem{exa}[thm]{Example}
\newtheorem*{assa}{Assumption A}
\newtheorem*{algstate}{Algebraic form of Theorem \ref{thmstattrainv}}

\newtheorem*{dmlcon}{Dynamical Mordell-Lang Conjecture}
\newtheorem*{condml}{Dynamical Mordell-Lang Conjecture}
\newtheorem*{congb}{Geometric Bogomolov Conjecture}
\newtheorem*{facts}{Facts}

\newtheorem*{pdd}{P(d)}
\newtheorem*{bfd}{BF(d)}

\newtheorem*{probreal}{Realization problems}
\numberwithin{equation}{section}
\newcounter{elno}                
\def\points{\list
{\hss\llap{\upshape{(\roman{elno})}}}{\usecounter{elno}}}
\let\endpoints=\endlist
\newcommand{\SH}{\rm SH}
\newcommand{\Cov}{\rm Cov}
\newcommand{\Tan}{\rm Tan}
\newcommand{\res}{\rm res}
\newcommand{\Om}{\Omega}
\newcommand{\om}{\omega}
\newcommand{\La}{\Lambda}
\newcommand{\la}{\lambda}
\newcommand{\mc}{\mathcal}
\newcommand{\mb}{\mathbb}
\newcommand{\surj}{\twoheadrightarrow}
\newcommand{\inj}{\hookrightarrow}
\newcommand{\zar}{{\rm zar}}
\newcommand{\Exc}{{\rm Exc}}
\newcommand{\an}{{\rm an}}
\newcommand{\red}{{\rm red}}
\newcommand{\codim}{{\rm codim}}
\newcommand{\Supp}{{\rm Supp}}
\newcommand{\rank}{{\rm rank}}
\newcommand{\Ker}{{\rm Ker \ }}
\newcommand{\Pic}{{\rm Pic}}
\newcommand{\Der}{{\rm Der}}
\newcommand{\Div}{{\rm Div}}
\newcommand{\Hom}{{\rm Hom}}
\newcommand{\im}{{\rm im}}
\newcommand{\Spec}{{\rm Spec \,}}
\newcommand{\Nef}{{\rm Nef \,}}
\newcommand{\Frac}{{\rm Frac \,}}
\newcommand{\Sing}{{\rm Sing}}
\newcommand{\sing}{{\rm sing}}
\newcommand{\reg}{{\rm reg}}
\newcommand{\Char}{{\rm char\,}}
\newcommand{\Tr}{{\rm Tr}}
\newcommand{\ord}{{\rm ord}}
\newcommand{\id}{{\rm id}}
\newcommand{\NE}{{\rm NE}}
\newcommand{\Gal}{{\rm Gal}}
\newcommand{\Min}{{\rm Min \ }}
\newcommand{\Max}{{\rm Max \ }}
\newcommand{\Alb}{{\rm Alb}\,}
\newcommand{\Aff}{{\rm Aff}\,}
\newcommand{\GL}{{\rm GL}\,}        
\newcommand{\PGL}{{\rm PGL}\,}
\newcommand{\Bir}{{\rm Bir}}
\newcommand{\Aut}{{\rm Aut}}
\newcommand{\End}{{\rm End}}
\newcommand{\Per}{{\rm Per}\,}
\newcommand{\ie}{{\it i.e.\/},\ }
\newcommand{\niso}{\not\cong}
\newcommand{\nin}{\not\in}
\newcommand{\soplus}[1]{\stackrel{#1}{\oplus}}
\newcommand{\by}[1]{\stackrel{#1}{\rightarrow}}
\newcommand{\longby}[1]{\stackrel{#1}{\longrightarrow}}
\newcommand{\vlongby}[1]{\stackrel{#1}{\mbox{\large{$\longrightarrow$}}}}
\newcommand{\ldownarrow}{\mbox{\Large{\Large{$\downarrow$}}}}
\newcommand{\lsearrow}{\mbox{\Large{$\searrow$}}}
\renewcommand{\d}{\stackrel{\mbox{\scriptsize{$\bullet$}}}{}}
\newcommand{\dlog}{{\rm dlog}\,}    
\newcommand{\longto}{\longrightarrow}
\newcommand{\vlongto}{\mbox{{\Large{$\longto$}}}}
\newcommand{\limdir}[1]{{\displaystyle{\mathop{\rm lim}_{\buildrel\longrightarrow\over{#1}}}}\,}
\newcommand{\liminv}[1]{{\displaystyle{\mathop{\rm lim}_{\buildrel\longleftarrow\over{#1}}}}\,}
\newcommand{\norm}[1]{\mbox{$\parallel{#1}\parallel$}}
\newcommand{\boxtensor}{{\Box\kern-9.03pt\raise1.42pt\hbox{$\times$}}}
\newcommand{\into}{\hookrightarrow}
\newcommand{\image}{{\rm image}\,}
\newcommand{\Lie}{{\rm Lie}\,}      
\newcommand{\CM}{\rm CM}
\newcommand{\sext}{\mbox{${\mathcal E}xt\,$}}  
\newcommand{\shom}{\mbox{${\mathcal H}om\,$}}  
\newcommand{\coker}{{\rm coker}\,}  
\newcommand{\sm}{{\rm sm}}
\newcommand{\pgcd}{\text{pgcd}}
\newcommand{\trd}{\text{tr.d.}}
\newcommand{\tensor}{\otimes}
\newcommand{\hotimes}{\hat{\otimes}}
\newcommand{\prop}{{\rm prop}}
\newcommand{\CH}{{\rm CH}}
\newcommand{\tr}{{\rm tr}}
\newcommand{\e}{\rm SH}

\renewcommand{\iff}{\mbox{ $\Longleftrightarrow$ }}
\newcommand{\supp}{{\rm supp}\,}
\newcommand{\ext}[1]{\stackrel{#1}{\wedge}}
\newcommand{\onto}{\mbox{$\,\>>>\hspace{-.5cm}\to\hspace{.15cm}$}}
\newcommand{\propsubset}
{\mbox{$\textstyle{
\subseteq_{\kern-5pt\raise-1pt\hbox{\mbox{\tiny{$/$}}}}}$}}
\newcommand{\sA}{{\mathcal A}}
\newcommand{\sB}{{\mathcal B}}
\newcommand{\sC}{{\mathcal C}}
\newcommand{\sD}{{\mathcal D}}
\newcommand{\sE}{{\mathcal E}}
\newcommand{\sF}{{\mathcal F}}
\newcommand{\sG}{{\mathcal G}}
\newcommand{\sH}{{\mathcal H}}
\newcommand{\sI}{{\mathcal I}}
\newcommand{\sJ}{{\mathcal J}}
\newcommand{\sK}{{\mathcal K}}
\newcommand{\sL}{{\mathcal L}}
\newcommand{\sM}{{\mathcal M}}
\newcommand{\sN}{{\mathcal N}}
\newcommand{\sO}{{\mathcal O}}
\newcommand{\sP}{{\mathcal P}}
\newcommand{\sQ}{{\mathcal Q}}
\newcommand{\sR}{{\mathcal R}}
\newcommand{\sS}{{\mathcal S}}
\newcommand{\sT}{{\mathcal T}}
\newcommand{\sU}{{\mathcal U}}
\newcommand{\sV}{{\mathcal V}}
\newcommand{\sW}{{\mathcal W}}
\newcommand{\sX}{{\mathcal X}}
\newcommand{\sY}{{\mathcal Y}}
\newcommand{\sZ}{{\mathcal Z}}
\newcommand{\A}{{\mathbb A}}
\newcommand{\B}{{\mathbb B}}
\newcommand{\C}{{\mathbb C}}
\newcommand{\D}{{\mathbb D}}
\newcommand{\E}{{\mathbb E}}
\newcommand{\F}{{\mathbb F}}
\newcommand{\G}{{\mathbb G}}
\newcommand{\HH}{{\mathbb H}}
\newcommand{\LL}{{\mathbb L}}
\newcommand{\J}{{\mathbb J}}
\newcommand{\M}{{\mathbb M}}
\newcommand{\N}{{\mathbb N}}
\renewcommand{\P}{{\mathbb P}}
\newcommand{\Q}{{\mathbb Q}}
\newcommand{\R}{{\mathbb R}}
\newcommand{\T}{{\mathbb T}}
\newcommand{\U}{{\mathbb U}}
\newcommand{\V}{{\mathbb V}}
\newcommand{\W}{{\mathbb W}}
\newcommand{\X}{{\mathbb X}}
\newcommand{\Y}{{\mathbb Y}}
\newcommand{\Z}{{\mathbb Z}}
\newcommand{\bk}{{\mathbf{k}}}

\newcommand{\bp}{{\mathbf{p}}}
\newcommand{\ep}{\varepsilon}
\newcommand{\bbk}{{\overline{\mathbf{k}}}}
\newcommand{\Fix}{\mathrm{Fix}}

\newcommand{\tor}{{\mathrm{tor}}}
\renewcommand{\div}{{\mathrm{div}}}

\newcommand{\trdeg}{{\mathrm{trdeg}}}
\newcommand{\Stab}{{\mathrm{Stab}}}

\newcommand{\OK}{{\overline{K}}}
\newcommand{\ok}{{\overline{k}}}

\newcommand{\cf}{{\color{red} [c.f. ?]}}
\newcommand{\jy}{\color{red} jy:}

\title[]{Homoclinic orbits, multiplier spectrum and rigidity theorems in complex dynamics}

\author{Zhuchao Ji}

\address{Institute for Theoretical Sciences, Westlake University, Hangzhou 310030, China}

\email{jizhuchao@westlake.edu.cn}

\author{Junyi Xie}


\address{Beijing International Center for Mathematical Research, Peking University, Beijing 100871, China}

\email{xiejunyi@bicmr.pku.edu.cn}


\date{\today}

\bibliographystyle{alpha}

\begin{abstract}
The aims of this paper are answering several conjectures and questions about multiplier spectrum  of rational maps and giving new proofs of several  rigidity theorems in complex dynamics, by combining tools from complex and non-archimedean dynamics. 
\par A remarkable theorem due to McMullen asserts that aside from the flexible Latt\`es family, the multiplier spectrum of periodic points determines the conjugacy class of rational maps up to finitely many choices. The proof relies on Thurston's rigidity theorem for post-critically finite maps, in where Teichm\"{u}ller theory is an essential tool.  We will give a new proof of McMullen's theorem (hence a new proof of Thurston's theorem) without using quasiconformal maps or Teichm\"{u}ller theory.
\par We show that aside from the flexible Latt\`es family, the length spectrum of periodic points determines the conjugacy class of rational maps up to finitely many choices.  This generalize the aforementioned McMullen's theorem. We will also prove a rigidity theorem for marked length spectrum. Similar ideas  also yield a simple proof of a rigidity theorem due to Zdunik. 
\par We show that a rational map is exceptional if and only if one of the following holds (i) the multipliers of periodic points are contained in the integer ring of an imaginary quadratic field;  (ii) all but finitely many periodic points have the same Lyapunov exponent.  This solves two conjectures of Milnor. 
\end{abstract}
\maketitle
\tableofcontents

\section{Introduction}
\subsection{Exceptional endomorphisms}
Let $f:\P^1\to \P^1$ be an endomorphism over $\C$ of degree at least $2$. It is called \emph{Latt\`es} if it is semi-conjugate to an endomorphism on an elliptic curve.
Further it is called \emph{flexible Latt\`es} if it is semi-conjugate to the multiplication by an integer  $n$ on an elliptic curve for some $|n|\geq 2.$ 
We say that $f$ is of \emph{monomial type} if it semi-is conjugate to the map $z\mapsto z^n$ on $\P^1$ for some $|n|\geq 2.$ 
We call $f$ \emph{exceptional} if it is Latt\`es or of monomial type. An endomorphism $f$ is exceptional if and only if some iterate  $f^n$ is exceptional.
 Exceptional endomorphisms are considered as the exceptional examples in complex dynamics.

\medskip

In this paper we will prove a criterion for an endomorphism being exceptional via the information of a \emph{homoclinic orbit} of $f$.
See Theorem \ref{thmadjconmultexcep} for the precise statement and  see Section \ref{sectionhomoclinic} for the definition and basic properties of   homoclinic orbits.
Since every $f$ has plenty of homoclinic orbits, the above 
criterion turns out to be very useful.   A direct consequence is the following characterization of exceptional endomorphisms by the linearity of a \emph{conformal expending repeller}(CER). 

\begin{thm}\label{linearcer}
Let $f:\P^1\to \P^1$ be an endomorphism over $\C$. Assume that $f$ has a linear CER which is not a finite set, then $f$ is exceptional. 
\end{thm}
 CER is an impotent concept  in complex dynamics introduced by Sullivan  \cite{sullivan1986quasiconformal}. See Section \ref{subsectioncer} for its definition and basic properties.

\subsection{Rigidity of stable algebraic families}
For $d\geq 1,$ let $\text{Rat}_d(\C)$ be the space of degree $d$ endomorphisms on $\P^1(\C)$.  
It is a smooth quasi-projective variety of dimension $2d+1$ \cite{Silverman2012}. 
Let $FL_d(\C)\subseteq \text{Rat}_d(\C)$ be the locus of flexible Latt\`es maps, which is Zariski closed in $\text{Rat}_d(\C)$.
The group $\PGL_2(\C)= \Aut(\P^1(\C))$ acts on $\text{Rat}_d(\C)$ by conjugacy. The geometric quotient 
$$\sM_d(\C):=\text{Rat}_d(\C)/\PGL_2(\C)$$ is the (coarse) \emph{moduli space} of endomorphisms  of degree $d$ \cite{Silverman2012}.
The moduli space $\sM_d(\C)=\Spec (\sO(\text{Rat}_d(\C)))^{\PGL_2(\C)}$ is an affine variety of dimension $2d-2$ \cite[Theorem 4.36(c)]{Silverman2007}.
Let $\Psi: \text{Rat}_d(\C)\to \sM_d(\C)$ be the quotient morphism.

\medskip

An \emph{irreducible algebraic family} $f_{\La}$ (of degree $d$ endomorphisms) is an algebraic endomorphism $f_{\La}: \P^1_{\La}\to \P^1_{\La}$ over an irreducible variety $\La$, such that for every $t\in \La(\C)$, the restriction $f_t$ of $f_{\La}$ above $t$ has degree $d.$
Giving an algebraic family over $\La$ is the same as giving an algebraic morphism $\La\to \text{Rat}_d.$
A family $f_{\La}$ is called \emph{isotrivial} if $\Psi(\La)$ is a single point.

\medskip

For every $f\in \text{Rat}_d(\C)$ and $n\geq 1$, $f^n$ has exactly $N_n:=d^n+1$ fixed points counted with multiplicity. 
Their multipliers define a point  $s_n(f)\in \C^{N_n}/S_{N_n}$\footnote{Via the symmetric polynomials, we have $\C^{N_n}/S_{N_n}\simeq \C^{N_n}.$}, where $S_{N_n}$ is the symmetric group which acts on $\C^{N_n}$ by permuting the coordinates.
The \emph{multiplier spectrum} of $f$ is the sequence $s_n(f), n\geq 1.$ An irreducible algebraic  family  is called \emph{stable} if  the multiplier spectrum of $f_t$ does not depend on $t\in \La(\C)$.\footnote{Stability has several equivalent definitions and can be defined for more general families \cite[Chapter 4]{mcmullen2016complex}.} 

\medskip

In 1987, McMullen \cite{McMullen1987} established the following remarkable characterization of stable  irreducible algebraic families:
\begin{thm}[McMullen]\label{thmmcmullen}
	Let $f_{\La}$ be a non-isotrivial stable irreducible algebraic family of degree $d\geq 2$, then  $f_t\in FL(\C)$ for every $t\in \La(\C)$. 
\end{thm}

McMullen's proof relies on the following Thurston's rigidity theorem for \emph{post-critically finite} (PCF) maps \cite{Douady1993}, in where Teichm\"{u}ller theory is essentially used. An endomorphism $f$ of degree $d\geq 2$ is called PCF if the critical orbits of $f$ is a finite set. 

\begin{thm}[Thurston]\label{thmthurston}
Let $f_{\La}$ be a non-isotrivial  irreducible algebraic family of PCF maps, then  $f_t\in FL(\C)$ for every $t\in \La(\C)$.
\end{thm}
	
In this paper, we will give a new proof of McMullen's theorem without using  quasiconformal maps or Teichm\"{u}ller theory.
Since an irreducible algebraic family of PCF maps is automatically stable, this leads to a new proof of Theorem \ref{thmthurston} without using quasiconformal maps or Teichm\"{u}ller theory. 
Except Theorem \ref{thmadjconmultexcep} whose proof relies on some basic complex analysis,
our proof of Theorem \ref{thmmcmullen} only requires some basic knowledges in Berkovich dynamics and hyperbolic dynamics.
We explain our strategy of the proof as follows: 

\medskip
Cutting by hypersurfaces, one may reduce to the case that $\La$ is a smooth affine curve. Let $W$ be the smooth projective compactification of $\La$ and let
$B:=W\setminus \La$. For every $o\in B$, our family induces a non-archimedean dynamical system on the Berkovich projective line (see Section \ref{sectionberko} for details),
which encodes the asymptotic behavior of $f_t$ when $t\to o.$ Since $f_{\La}$ is  non-isotrivial and stable, via the study of non-archimedian dynamics, we show that there is one point $o\in B$ such that when $t\to o$, $f_t$ ``degenerates" to a map taking form $z\mapsto z^m$ in a suitable coordinate, where $2\leq m\leq d-1.$ 
The above degeneration $z\mapsto z^m$ is called a \emph{rescaling limit} of $f_{\La}$ at $o$, in the sense of Kiwi \cite{kiwi2015rescaling}, see Definition \ref{defiresclim}. On the central fiber, it is easy to find a homoclinic orbit satisfying the condition in our exceptional criterion Theorem \ref{thmadjconmultexcep}.  Using an argument in hyperbolic dynamics \cite{jonsson1998holomorphic} (see Lemma \ref{motion}), 
we can deform such homoclinic orbit to nearby fibers preserving the required condition. By Theorem \ref{thmadjconmultexcep}, $f_t$ is exceptional for all $t$ sufficiently close to $o$.
We conclude the proof by the fact that  exceptional endomorphisms that are not flexible Latt\`es are isolated in the moduli space $\sM_d(\C)$.

\subsection{Length spectrum as moduli}
According to the Noetheriality of the Zariski topology on $\text{Rat}_d$,  McMullen's rigidity theorem can be reformulated as follows:
\begin{thm}[Multiplier spectrum as moduli=Theorem \ref{thmmcmullen}]\label{thmmcmmulti}
Aside from the flexible Latt\`es family, the multiplier spectrum determines the conjugacy class of endomorphisms in $\text{Rat}_d(\C)$, $d\geq 2$, up to finitely many choices.
\end{thm}

Replace the multipliers by its norm in the definition of multiplier spectrum, one get the definition of the \emph{length spectrum}.
More precisely, for every $f\in \text{Rat}_d(\C)$ and $n\geq 1$, we denote by $L_n(f)\in \R^{N_n}/S_{N_n}$ the element corresponding to the multiset $\{|\la_1|,\dots, |\la_{N_n}|\},$ where 
$\la_i, i=1,\dots, N_n$ are the multipliers of all $f^n$-fixed points. The length spectrum of $f$ is defined to be the sequence $L_n(f), n\geq 1.$  A priori,  the length spectrum contains less information than the multiplier spectrum. 
But in this paper we will show that it determines the conjugacy class up to finitely many choices,  hence generalize Theorem \ref{thmmcmmulti}.
\begin{thm}[Length spectrum as moduli]\label{thmlength}
Aside from the flexible Latt\`es family, the length spectrum determines the conjugacy class of endomorphisms in $\text{Rat}_d(\C)$, $d\geq 2$, up to finitely many choices.
\end{thm}

\subsubsection{Strategy of the proof of Theorem \ref{thmlength}:}
Let $g\in \text{Rat}_d(\C)\setminus FL_d(\C)$. We need to show that the image of $$Z:=\{f\in \text{Rat}_d(\C)\setminus FL_d(\C) |\,\, L(f)=L(g)\}$$ in $\sM_d(\C)$ is finite.
For $n\geq 0$, set  $$Z_n:=\{f\in \text{Rat}_d(\C)\setminus FL_d(\C) |\,\, L_i(f)=L_i(g), i=1,\dots,n \}.$$ It is clear that $Z_i, i\geq 1$
 is a decreasing sequence of closed subsets of $\text{Rat}_d(\C)\setminus FL_d(\C)$ and 
 $Z=\cap_{n\geq 1} Z_n.$
For simplicity, we assume that all periodic points of $g$ are repelling. 
Otherwise, instead of the length spectrum $L(g)$ of all periodic points, we consider the length spectrum $RL(g)$ of all repelling periodic points. 
Such a change only adds some technical difficulties. 
To get a contradiction, we assume that $\Psi(Z)\in \sM_d(\C)$ is infinite.
 Our proof contains two steps.

\medskip

In Step 1, we show that $Z=Z_N$ for some $N\geq 0.$  
We first look at the analogue of this step for multiplier spectrum.
The analogue of $Z_n$ is
$$Z'_n:=\{f\in \text{Rat}_d(\C)\setminus FL_d(\C) |\,\, s_i(f)=s_i(g), i=1,\dots,n \},$$
which  is Zariski closed in $\text{Rat}_d(\C)\setminus FL_d(\C).$
Hence $Z_n'$ is stable when $n$ large by the Noetheriality. This is how Theorem \ref{thmmcmullen} implies Theorem \ref{thmmcmmulti}.
In the  length spectrum case,
since the $n$-th length map $L_n: \text{Rat}_d(\C)\to \R^{N_n}/S_{N_n}$ takes only real values, it is more natural to view $\text{Rat}_d(\C)$ as a real algebraic variety by splitting the complex variable to two real variables via $z=x+iy$. 
If all $Z_n, n\geq 1$ are real algebraic, we can still conclude this step by the Noetheriality.
Unfortunately, this is not true in general (c.f. Theorem \ref{thmexamplenonalg}).
Since the map $L_n^2$ sending $f$ to $\{|\la_1|^2,\dots, |\la_{N_n}|^2\}\in \R^{N_n}/S_{N_n}$ is semialgebraic, all $Z_n, n\geq 1$ are semialgebraic.  
The problem is that, in general closed semialgebraic sets do not satisfy the descending chain condition. 
We solve this problem by introducing a class of closed semialgebraic sets called  \emph{admissible} subsets.
Roughly speaking, admissible subsets are the closed subsets which are images of algebraic subsets under \'etale morphisms. See Section \ref{subsectionimageetale} for its precise definition and basic properties. We show that admissible subsets satisfy the descending chain condition.
Under the assumption that all periodic points of $g$ are repelling, we can show that all $Z_n$ are admissible.  The admissibility is only used to prove Theorem \ref{thmlength}. 

\medskip

Step 1 implies that $Z=Z_N$ is semialgebraic. Since $\Psi(Z)$ is infinite, there is an analytic curve $\gamma\simeq [0,1]$ contained in $Z$ such that $\Psi(\gamma)$ is not a point.
Every $t\in \gamma\subseteq \text{Rat}_d$ defines an endomorphism $f_t$. After shrinking $\gamma$, we may assume that $f_0$ is not exceptional.

\medskip

In Step 2, we show that the multiplier spectrum of $f_t$ does not depend on $t\in \gamma.$ Once Step 2 is finished, we get a contradiction by Theorem \ref{thmmcmmulti}.
Since for every $t\in \gamma$, $L(f_t)=L(g)$, all periodic points of $f_t$ are repelling. For every repelling periodic point $x$ of $f_0$, there is a real analytic map $\phi_x: \gamma\to \P^1(\C)$ such that 
for every $t\in \gamma$, $\phi_x(t)$ and $x$ have the same minimal period and  the norms of their multipliers are same.
Using homoclinic orbits, we may construct a CER $E_0$ of $f_0$ containing  $x.$ It is non-linear by Theorem \ref{linearcer}. By Lemma \ref{motion}, for $t$ sufficiently small, $E_0$ can be deformed to a CER $E_t$ of $f_t$ containing $\phi_x(t).$ Using Sullivan's rigidity theorem \cite{sullivan1986quasiconformal} (Theorem \ref{cerrigidity}), we show that $E_0$ and $E_t$ are conformally conjugate. In particular, the multipliers of $\phi_x(t)$  is a constant for $t$ sufficiently small. Since $\gamma$ is real analytic, the multipliers of $\phi_x(t)$  is a constant on $\gamma.$ Since $x$ is arbitrary,  all $f_t,t\in \gamma$ have the  same multiplier spectrum.  This finishes Step 2.

\subsubsection{Further discussion}  
It is interesting to know whether the uniform version of Theorem \ref{thmlength} holds.
\begin{que}\label{queuniformlen}
Is there $N\geq 1$ depending only on $d\geq 2$, such that for every $f\in \text{Rat}_d(\C)\setminus FL_d(\C)$, 
 $$\#\Psi(\{g\in \text{Rat}_d(\C)\setminus FL_d(\C)|\,\, L_i(g)=L_i(f), i=1,\dots,N\})\leq N\,?$$
\end{que}

For every $n\geq 0,$ we set $$R_n:=\{(f,g)\in (\text{Rat}_d(\C)\setminus FL_d(\C))^2|\,\, L_i(f)=L_i(g), i=1,\dots, n\}$$ and 
$$R_n':=\{(f,g)\in (\text{Rat}_d(\C)\setminus FL_d(\C))^2|\,\, s_i(f)=s_i(g),  i=1,\dots, n\}.$$
Both of them are  decreasing closed subsets of $(\text{Rat}_d(\C)\setminus FL_d(\C))^2.$ 
Since all $R_n'$ are algebraic subsets of $(\text{Rat}_d(\C)\setminus FL_d(\C))^2$, the sequence $R_n'$ is stable for $n$ large.
This implies that Theorem \ref{thmmcmmulti} for multiplier spectrum is equivalent to its uniform version.

If one can show that the sequence 
$R_n, n\geq 0$ is stable, for example if one can show that $R_n$ are admissible, 
then Question \ref{queuniformlen} has a positive answer. 
But at the moment, we only know that $R_n$ are semialgebraic.

\subsection{Marked multiplier and length spectrum}
By Theorem \ref{thmlength} and \ref{thmmcmmulti},
aside from the flexible Latt\`es family, the length spectrum (hence the multiplier spectrum) determines the conjugacy class of endomorphisms of degree $d\geq 2$ \emph{up to finitely many choices}. On the other hand, by \cite[Theorem 6.62]{Silverman2007}, the multiplier spectrum $f\mapsto s(f)$
(hence the length spectrum $f\mapsto L(f)$) is far from being injective when $d$ large.
For this reason, we consider the \emph{marked multiplier and length spectrum}.
We show that both of them are rigid. 

\begin{thm}[Marked multiplier spectrum rigidity]\label{markedmultiplier}
	Let $f$ and $g$ be two endomorphisms of $\P^1$ over $\C$ of degree at least $2$ such that $f$ is not exceptional.  Assume there is a homeomorphism $h:\sJ(f)\to \sJ(g)$ such that $h\circ f=g\circ h$ on $\sJ(f)$. Then the following two conditions are equivalent.
	 \begin{points}
	\item There is a non-empty open set $\Omega \subset \sJ(f)$ such that for every periodic point $x\in \Omega$ we have $df^n(x)=dg^n(h(x))$, where $n$ is the period of $x$;
	\item $h$  extends to an automorphism $h:\P^1(\C)\to \P^1(\C)$ such that $h\circ f=g\circ h$ on $\P^1(\C)$.
	\end{points}	
\end{thm}

Let $U,V\subset \P^1(\C)$ be two open sets.  A homeomorphism $h:U\to V$ is called \emph{conformal} if  $h$ is holomorphic or antiholomorphic in every connected component of $U$. Note that a conformal map $h$ is holomorphic if and only if $h$ preserves the orientation of $\P^1(\C)$. 

\begin{thm}[Marked length spectrum rigidity]\label{markedlength}
	Let $f$ and $g$ be two endomorphisms of $\P^1$ over $\C$ of degree at least $2$ such that $f$ is not exceptional.  Assume there is a homeomorphism $h:\sJ(f)\to \sJ(g)$ such that $h\circ f=g\circ h$ on $\sJ(f)$. Then the following two conditions are equivalent.
	 \begin{points}
	\item There is a non-empty open set $\Omega \subset \sJ(f)$ such that for every periodic point $x\in \Omega$ we have $|df^n(x)|=|dg^n(h(x))|$, where $n$ is the period of $x$;
 
	\item $h$  extends to a conformal map $h:\P^1(\C)\to \P^1(\C)$ such that $h\circ f=g\circ h$ on $\P^1(\C)$.
	\end{points}	
\end{thm}

Note that if $h:\Omega \to h(\Omega)$ is bi-Lipschitz, then it is not hard to show that  for $n$-periodic point $x\in \Omega$ we have $|df^n(x)|=|dg^n(h(x))|$. So the above theorem implies that a locally bi-Lipschitz conjugacy can be improved to a conformal conjugacy on $\P^1(\C)$. 

\medskip

Combining Theorem \ref{markedmultiplier} and  $\la$-Lemma \cite[Theorem 4.1]{mcmullen2016complex}, we get a second proof of Theorem \ref{thmmcmullen}. This proof does not use  Teichm\"{u}ller theory, but we need to use quasiconformal maps and the highly non-trivial Sullivan's rigidity theorem, which is a great achievement in thermodynamic formalism.

\begin{rem}

In Theorem \ref{markedlength}, in general $h$ can not be extended to an  automorphism on $\P^1(\C)$. The complex conjugacy 
$\sigma: z\mapsto \overline{z}$ induces a mark $h:=\sigma|_{\sJ(f)}: \sJ(f)\to \overline{\sJ(f)}=\sJ(\overline{f})$ preserving the length spectrum.
In general, some periodic point of $f$ may have non-real multipliers, hence in this case $h$  can not be extended to an  automorphism on $\P^1(\C)$. 
\end{rem}

\begin{rem} Theorem \ref{markedlength} was proved by Przytycki and Urbanski in \cite[Theorem 1.9]{przytycki1999rigidity} under the assumptions that both $f$ and $g$ are tame and $\Omega=\sJ(f)$. See \cite[Definition 1.1]{przytycki1999rigidity} for the precise definition of tameness. In \cite[Theorem 2]{Rees1984}, Rees showed that there are endomorphisms having dense critical orbits, hence not tame. 
\end{rem}
\medskip
\par The study of marked length spectrum rigidity has been investigated in various settings in dynamics and geometry.

\par In one-dimensional real dynamics, marked multiplier spectrum rigidity was proved for expanding circle maps (see Shub-Sullivan \cite{shub1985expanding}), and for some unimodal maps (see Martens-de Melo \cite{martens1999multipliers} and Li-Shen \cite{li2006smooth}).

\par In the contexts of geodesic flows on Riemannian manifolds with negative curvature,  
a long-standing conjecture stated by Burns-Katok \cite{Burns1985} (and
probably considered even before) asserted
the rigidity of marked length spectrum (for closed geodesics).  The surface case was proved by Otal \cite{Otal1990} and by Croke \cite{Croke1990} independently. 
A Local version of the Burns-Katok conjecture in any dimension
was proved by Guillarmou-Lefeuvre \cite{Guillarmou2019}.

\par It was also studied in dynamical billiards.  We refer the readers to Huang-Kaloshin-Sorrentino \cite{Huang2018}, B\'alint-De Simoi-Kaloshin-Leguil \cite{Balint2020}, De Simoi-Kaloshin-Leguil \cite{DeSimoi2019},  and the references therein.

%

\medskip

\par We prove Theorem \ref{markedlength} by combining Theorem \ref{linearcer} and Sullivan's rigidity theorem \cite{sullivan1986quasiconformal}, see Theorem \ref{cerrigidity}.
More precisely, let $o$ be a repelling fixed point of $f$.
We construct a family $C$ of CERs of  $f$ using homoclinic orbits which covers all backward orbits of $o$.
By Theorem \ref{linearcer}, all of them are non-linear.  We show that their images under $h$ are CERs of $g$. Applying Sullivan's rigidity theorem, we get conformal conjugacies link the CERs in $C$ to their images. Two CERs in $C$ have ``large" intersections,
hence those conformal conjugacies can be patched together. Using this, we  get a conformal extension of $h$ to some disk intersecting the Julia set of $f$. We can further extend it to a global conformal map on $\P^1(\C).$

Theorem \ref{markedmultiplier} is a simple consequence of Theorem \ref{markedlength} and a result of Eremenko-van Strien \cite[Theorem 1]{eremenko2011rational}  about endomorphisms with real multipliers.

\subsection{Zdunik's rigidity theorem}

The following rigidity theorem was proved by Zdunik \cite{zdunik1990parabolic}. 
\begin{thm}[Zdunik]\label{zdunik}
Let $f:\P^1\to \P^1$ be an endomorphism over $\C$ of degree at least $2$. Let $\mu$ be the maximal entropy measure and let $\alpha$ be the Hausdorff dimension of $\mu$. Then $\mu$ is absolutely continous with respect to the $\alpha$-dimensinal Hausdorff measure $\Lambda_\alpha$ on the Julia set if and only if $f$ is exceptional. 
\end{thm}
Zdunik's  proof is divided into two parts. The first part was proved in her previous work \cite[Theorem 6]{przytycki1989harmonic} with Przytycki and Urbanski.
Later, she proved the second part (hence Theorem \ref{zdunik}) in \cite{zdunik1990parabolic}.
In this paper we will give a simple proof of the second part using Theorem \ref{linearcer}.

\subsection{Milnor's conjectures on multiplier spectrum}
As applications of Theorem \ref{thmadjconmultexcep} and Theorem \ref{linearcer}, we prove two conjectures of Milnor proposed in \cite{milnor2006lattes}.

\begin{thm}\label{thmmilnor}Let $f:\P^1\to \P^1$ be an endomorphism over $\C$ of degree at least $2$. Let $K$ be an imaginary quadratic field. Let  $O_K$ be the ring of integers of $K$.  If for every $n\geq 1$ and every $n$-periodic point $x$ of $f$, $df^n(x)\in O_K$, then $f$ is exceptional.
\end{thm}

The  inverse of Theorem \ref{thmmilnor} is also true by Milnor \cite[Corollary 3.9 and Lemma 5.6]{milnor2006lattes}.
In fact, the original conjecture of Milnor concerns the case $K=\Q$.
Since imaginary quadratic fields exist (e.g. $\Q(i)$) and they contain $\Q$, Theorem \ref{thmmilnor} implies Milnor's original conjecture.

Some special cases of Milnor's conjecture for integer multipliers are known before by Huguin:
\begin{points}
\item In \cite{Huguin2022}, the conjecture was proved for quadratic endomorphisms. 
\item In \cite{Huguin2021}, the conjecture was proved for unicritical polynomials.
In fact, Huguin proved a stronger statement, which only assumes the multipliers are in $\Q$ (instead of $\Z$).
\end{points}

\rem In the recent preprint \cite{Huguin2022a}, Huguin reproved and strengthened our Theorem \ref{thmmilnor}. In his result,
the multipliers are only assumed to be contained in an arbitrary number field.
Huguin's result relies on an arithmetic equidistribution result for small points proved by Autissier \cite{Autissier2001} and on a characterization of exceptional maps proved by Zdunik \cite{zdunik2014characteristic}.
\endrem

%
\medskip

The following result confirms another conjecture of Milnor in \cite{milnor2006lattes}.
\begin{thm}\label{thmmilnor1}
Let $f:\P^1\to \P^1$ be an endomorphism over $\C$ of degree at least $2$.  Assume there exist $a>0$ such that for every but finitely many periodic point $x$,  $f^n(x)=x$,  we have $|df^n(x)|=a^n$. Then $f$ is exceptional.
\end{thm}

\begin{rmk}
Theorem \ref{thmmilnor1} can also be deduced by a minor modification of an argument of Zdunik \cite{zdunik2014characteristic}. 
\end{rmk}

Let $x$ be a $n$-periodic point of $f$. The \emph{Lyapunov exponent} of $x$ is a real number defined by $\frac{1}{n}\log |df^n(x)|$.  We let $\Delta(f)$ be the closure of the Lyapunov exponents of periodic points contained in the Julia set. Combining Theorem \ref{thmmilnor1} and results due to Gelfert-Przytycki-Rams-Rivera Letelier \cite{gelfert2010lyapunov}, \cite{gelfert2013lyapunov}, we get the following description of $\Delta(f)$ when $f$ is non-exceptional.  A closed interval in $\R$ is called non-trivial if it is not a singleton.
\begin{cor}\label{lyapunovspectrum}
Let $f:\P^1\to \P^1$ be a non-exceptional endomorphism over $\C$ of degree at least $2$. Then $\Delta(f)$ is a disjoint union of a non-trivial closed interval $I$ and a finite set $E$ (possibly empty).  Moreover there are at most $4$ periodic points whose Lyapunov exponents are contained in $E$, in particular $|E|\leq 4$.
\end{cor}

\subsection{Organization of the paper}
In Section \ref{sectionhomoclinic} we prove some basic properties of homoclinic orbits and we prove the fundamental exceptional criterion Theorem \ref{thmadjconmultexcep} by using only the information of a homoclinic orbit. In Section \ref{sectionmilnor} we prove Theorem \ref{thmmilnor}. In Section \ref{sectionberko} we recall some reslts about dynamics on the Berkovich projective line.  In Section \ref{sectionrescaling} we study the rescaling limit via the dynamics on the Berkovich projective line.  In Section \ref{sectionmcmullen} we give a new proof of McMullen's theorem (Theorem \ref{thmmcmullen}) by studying rescaling limits. In Section \ref{sectioncer} we recall some results about CER, and we prove Theorem \ref{linearcer}, Theorem \ref{markedmultiplier}, Theorem \ref{markedlength}, Theorem \ref{thmmilnor1} and Corollary \ref{lyapunovspectrum}.  Moreover we give a new proof of Theorem \ref{zdunik} and we give  another proof of Theorem \ref{thmmcmullen}.  In Section \ref{sectionlength} we prove Theorem \ref{thmlength}. 
\subsection*{Acknowledgement}
The authors would like to thank Serge Cantat, Romain Dujardin and Zhiqiang Li for their comments on the first
version of the paper. We thank Huguin for informing us his beautiful recent work \cite{Huguin2022a} for reproving and strengthening our Theorem \ref{thmmilnor}.
The first named author would like to thank the support and hospitality of BICMR during the preparation of this paper. The second named author would like to thank Thomas Gauthier and Gabriel Vigny for interesting discussions on Thurston's rigidity theorem.
 

\section{Homoclinic orbits and applications}\label{sectionhomoclinic}
For an endomorphism $f$ of $\P^1$ of degree at least $2$, we denote by $C(f)$ the set of critical points of $f$ and $PC(f):=\cup_{n\geq 1}f^n(C(f))$ the postcritical set. 
In this section, $\P^1(\C)$ is endowed with the complex topology.
\medskip

Let $f: \P^1\to \P^1$ an endomorphism over $\C$ of degree at least $2.$  Let $o$ be a repelling fixed point of $f$.
A \emph{homoclinic orbit} \footnote{This terminology was introduced by Milnor \cite{milnor2011dynamics} in his presentation of Julia's proof that repelling periodic points are dense in the Julia set. The word ``homoclinic orbit" dates back to Poincar\'e.} of $f$ at $o$ is a sequence of points $o_i, i\geq 0$ satisfying the following properties:
\begin{points}
\item $o_0=o$, $o_1\neq o$ and $f(o_{i})=o_{i-1}$ for $i\geq 1;$
\item $\lim\limits_{i\to \infty} o_i=o$.
\end{points}
Be aware that $o_i, i\geq 0$ is actually a backward orbit.

The main result of this Section is Theorem \ref{thmadjconmultexcep} which 
is a criterion for an endomorphism $f$ being exceptional via the information of a homoclinic orbit.
We will state and prove this theorem in the end of this section.

\subsection{Linearization domain and good return times}

Define a \emph{linearization domain} of $o$ to be an open neighborhood $U$ of $o$ such that there is an isomorphism $\phi: U\to \D$ sending $o$ to $0$, which  conjugates $f|_{U_0}:U_0\to U$ to the morphism $z\mapsto \la z$ via $\phi$, where  $U_0=f^{-1}(U)\cap U$ and $\la=df(o)$.  
We call such $\phi$ a \emph{linearization on $U$}. 


Define $g$ to be the morphism $U\to U$ sending $z$ to $\phi^{-1}(\la^{-1}\phi(z))$. It is the unique endomorphism of $U$ satisfiying $f\circ g=\id.$

\begin{rem}
By Koenigs' theorem \cite[Theorem 8.2] {milnor2011dynamics}, for every repelling point $o$ there is always a linearization domain $U$ . 
For every $r\in (0,1]$, $\phi^{-1}(\D(0,r))$ is also a linearization domain of $o$. In particular, the linearization domains of $o$ form a neighborhood system of  of $o.$
\end{rem}
\rem
Since $g$ is injective, for every $x\in U$, $f^{-1}(x)\cap U=g(x).$ In particular, if $o_i\in U$ for $i \geq l$, then $o_i=g^{i-l}(o_l)$ for all $i\geq l.$
\endrem

The following lemma shows that for every repelling fixed point $o$, there are many homoclinic orbits.
\begin{lem}\label{lemexisthomo}For every integer $m\geq 0$ and for every $a\in f^{-m}(o)$, there is a homoclinic orbit $o_i, i\geq 0$ of $o$ such that $o_m=a.$
\end{lem}
\proof
Let $U$ be a linearization domain of $o.$
Since preimages of $a$ are dense in the Julia set, there is $l\geq m$ such that $f^{m-l}(a)\cap U\neq\emptyset.$
Pick $o_l\in f^{m-l}(a)\cap U$ and for $i=0,\dots, l$, set $o_i:=f^{l-i}(o_l).$ Then $o_0=o$ and $o_m=a.$
For $i\geq l+1$, set $o_i:=g^{i-l}(o_l)$. Then $o_i, i\geq 0$ is a homoclinic orbit of $o.$
\endproof

\begin{defi}
Let $U$ be a connected open neighborhood of $o$ such that $U$ is contained in a linearization domain. For $i\geq 0$ let $U_i$ be the connected component of $f^{-i}(U)$ containing $o_i$. An integer $m\geq 1$ is called a  \emph{good return time} for the homoclinic orbit and $U$ if 
\begin{points}
\item $o_i\in U$ for every $i\geq m$;
\item $U_m\subset\subset U$, and 
and $f^m:U_m\to U$ is an isomorphism between $U_m$ and $U.$  	
\end{points}
\end{defi}

\rem\label{remhorsmenlarge} If $U$ itself is a linearization domain and $m$ is a good return time then $i$ is a good return time for all $i\geq m$. Indeed, one has $U_i=g^{i-m}(U_m)\subset\subset U$ and $f^i:U_i\to U$ can be writen as $f^m\circ g^{m-i}$ which is an isomorphism.
\endrem

\begin{pro}\label{proexihors}
The following statements are equivalent:
\begin{points}
\item $o_i\not\in C(f)$ for every $i\geq 1$;
\item there is a linearization domain $U$ and an integer $m\geq 1$ which is a good return time of $U$;
\item there is a linearization domain $U$ such that for every connected open neighborhood $V$ of $o$ , $V\subset U$, there is an integer $m\geq 1$ which is a good return time of $V$.

\end{points}

In particular, when $o\not\in PC(f)$, (i) (hence (ii) and (iii)) are satisfied.
\end{pro}
\proof
We first show (i) is equivalent to (ii). To see (ii) implies (i), let $m$ be a good return time of $U$, then by the definition of good return time $o_i\not\in C(f)$ for $i=1,\dots,m$. By Remark \ref{remhorsmenlarge} we conclude that $o_i\notin C(f)$ for every $i\geq 1$. To see (i) implies (ii), first choose a linearization domain $U_0$.
Let $g:U_0\to U_0$ be the morphism such that $f\circ g=\id$. 
Since $\lim\limits_{i\to \infty} o_i=o$, there is $l\geq 1$ such that 
$o_i\in U_0$ for $i\geq l.$ Since $o_i\not\in C(f)$ for every $i\geq 1$, we have $d(f^l)(o_l)\neq 0.$ So there is an open neighborhood $W$ of $o_l$ in $U_0$ such that 
$f^l(W)\subseteq U_0$ and $f^l|_W$ is injective. 
Pick a linearization domain of $U$ of $o$ contained in $f^l(W)$. Set $U_l:=f^{-l}(U) \cap W.$ Since $g$ is attracting, there is $m\geq l$ such that $g^{m-l}(U_l)\subset\subset U.$
We note that $U_m:=f^{-m}(U)\cap U=g^{m-l}(U_l)$. Hence $U_m\subset\subset U$, and $f^m:U_m\to U$ is an isomorphism. This implies (ii).
\medskip

It is clear that (iii) implies (ii).  It remains to show (ii) implies (iii). Let $l\geq 1$ be a good return time of $U$. Let $U_i$ (resp. $V_i$) be the connected component of $f^{-i}(U)$ (resp. $f^{-i}(V)$) for $i\geq 0$.  We have $U_l\subset\subset U$. Since $g$ is attracting, there is $m\geq l$ such that $g^{m-l}(U_l)\subset\subset V.$ This implies $m$ is a good return time of $V$.
 
\endproof
\subsection{Adjoint sequence of periodic points}
Let $U$ be a linearization domain and let $m$ be a good return time of $U$. We construct a sequence of periodic points $q_i, i\geq m$ as follows:
By Remark \ref{remhorsmenlarge}, for every $i\geq m$, 
$f^i|_{U_i}: U_i\to U$ is an isomorphism. Since $U_i\subset\subset  U$, the morphism $(f^i|_{U_i})^{-1}: U\to U_i$ is strictly attracting with respect to the hyperbolic metric on $U$. Hence it has a unique fixed point $q_i\in U_i$. Such $q_i$ is the unique $i$-periodic point of $f$ which is contained in $U_i.$ Indeed, $i$ is the smallest period of $q_i$ and $q_i$ is repelling.
We call such a sequence an \emph{adjoint sequence} for  the homoclinic orbit $o_i, i\geq 0$ with respect to the linearization domain $U$ and the good return time $m$ (we write $(U,m)$ for short). 
One say that a sequence of points  $q_i, i\geq 0$ is  an \emph{adjoint sequence}  of the homoclinic orbit $o_i, i\geq 0$ if $q_i, i\geq m$ is an adjoint sequence for $o_i, i\geq 0$ with respect to some $(U,m).$
It is clear that for every adjoint sequence $q_i, i\geq 0$ of $o_i, i\geq 0$, $\lim\limits_{i\to \infty} q_i=o.$
The following lemma shows that the adjoint sequences are unique modulo finite terms. 

\begin{lem}\label{lemadjunique}Let $q_i, i\geq 0$ and $q_i', i\geq 0$ be two adjoint sequence for $o_i, i\geq 0$. Then there is $l\geq 0$ such that $q_i=q_i'$ for all $i\geq l$.
\end{lem}
\proof We only need to prove the case where $q_i, i\geq l$ is an adjoint sequence with respect to  $(U, l)$  and $q_i', i\geq l'$ is an adjoint sequence with respect to $(U', l').$
Since there is a linearization domain $U''$ such that $U''\subseteq U\cap U'$, we may assume that $U'\subseteq U.$
After replacing $l,l'$ by $\max\left\{l,l'\right\}$, we may assume that $l=l'.$ Then for every $i\geq l$, $U'_i\subseteq U_i.$ Then both $q_i$ and $q_i'$ are the unique $i$-periodic point of $f$ in $U_i$.
So $q_i=q_i'$ for $i\geq l.$
\endproof

\subsection{Poincar\'e's linearization map}
 Set $\la:=df(o)\in \C$. Since $o$ is repelling, $|\la|>1.$
 Let $(U,m)$ be the pair of linearization domain and good return time for $o_i, i\geq 0$ and let $q_i, i\geq 0$ be an adjoint sequence.

 A theorem of Poincar\'e \cite[Corollary 8.12] {milnor2011dynamics} says that there is a morphism $\psi: \C\to \P^1(\C)$ such that $\psi|_{\D}$ gives an isomorphism between $\D$ and $U$ and 
 \begin{equation}\label{equationpoincare}f(\psi(z))=\psi(\la z)
 \end{equation} for every $z\in \C.$
In particular, $\psi|_{\D}^{-1}: U\to \D$ is a linearization of $f$ on $U.$
Such a $\psi$ is called a \emph{Poincar\'e map}.

The following criterion  for exceptional endomorphisms using the Poincar\'e map $\psi$ is due to Ritt.
\begin{lemma}\label{lemritt}\cite{ritt1922periodic}
	If the Poincar\'e map $\psi$ is periodic, i.e. there is a $a\in\C^*$ such that $\psi(z+a)=\psi(z)$ for every $z\in\C$,  then $f$ is exceptional. 
\end{lemma}

Ritt's theorem can be easily generalized as following. 
\begin{lem}\label{lemaffineconjex}If there is an affine automorphism $h: \C\to \C$ such that $h(0)\neq 0$ and $\psi\circ h= \psi$, then $f$ is exceptional.
\end{lem}
\proof
Let $G$ be the group of affine automorphisms $g$ of $\C$ satisfying $\psi\circ g=\psi.$ 
We have $h\in G$. It takes form $h: z\mapsto az+b$, $a\in \C^*$ and $b=h(0)\in \C^*.$
For every $z\in \C$, we have 
$$\psi(\la h(\la^{-1}z))=f\psi(h(\la^{-1}z))=f\psi(\la^{-1}z)=\psi(z).$$
Hence the automorphism
$g: z\mapsto \la h(\la^{-1}z)=az+\la b$ is contained in $G.$
Then $T:=h^{-1}\circ g: z\mapsto z+a^{-1}(\la-1)b$ is contained in $G.$
Since $b\neq 0$ and $|\la|>1$, $T$ is a nontrivial translation. 
We conclude the proof by Lemma \ref{lemritt}. 
\endproof
\medskip

Set $P:=\la^m \psi|_{\D}^{-1}(o_m)$ and $V:=\la^m(\psi|_{\D}^{-1}(U_m)).$ 
For $i\geq m,$ set $Q_i:=\psi|_{\D}^{-1}(q_i).$ 
One has $\psi(V)=U$, $\psi(P)=o$  and $\psi|_V: V\to U$ is an isomorphism. 
We set $T:=(\psi|_V)^{-1}\circ\psi|_{\D}:\D\to V$, then $T$  is an isomorphism. Similar constructions of $T$ appeared already in the works of Ritt \cite{ritt1922periodic} and Eremenko-van Strien \cite{eremenko2011rational}. 
We summaries our construction in the following figure.
\begin{figure}[h]
\centering
\includegraphics[width=11.5cm]{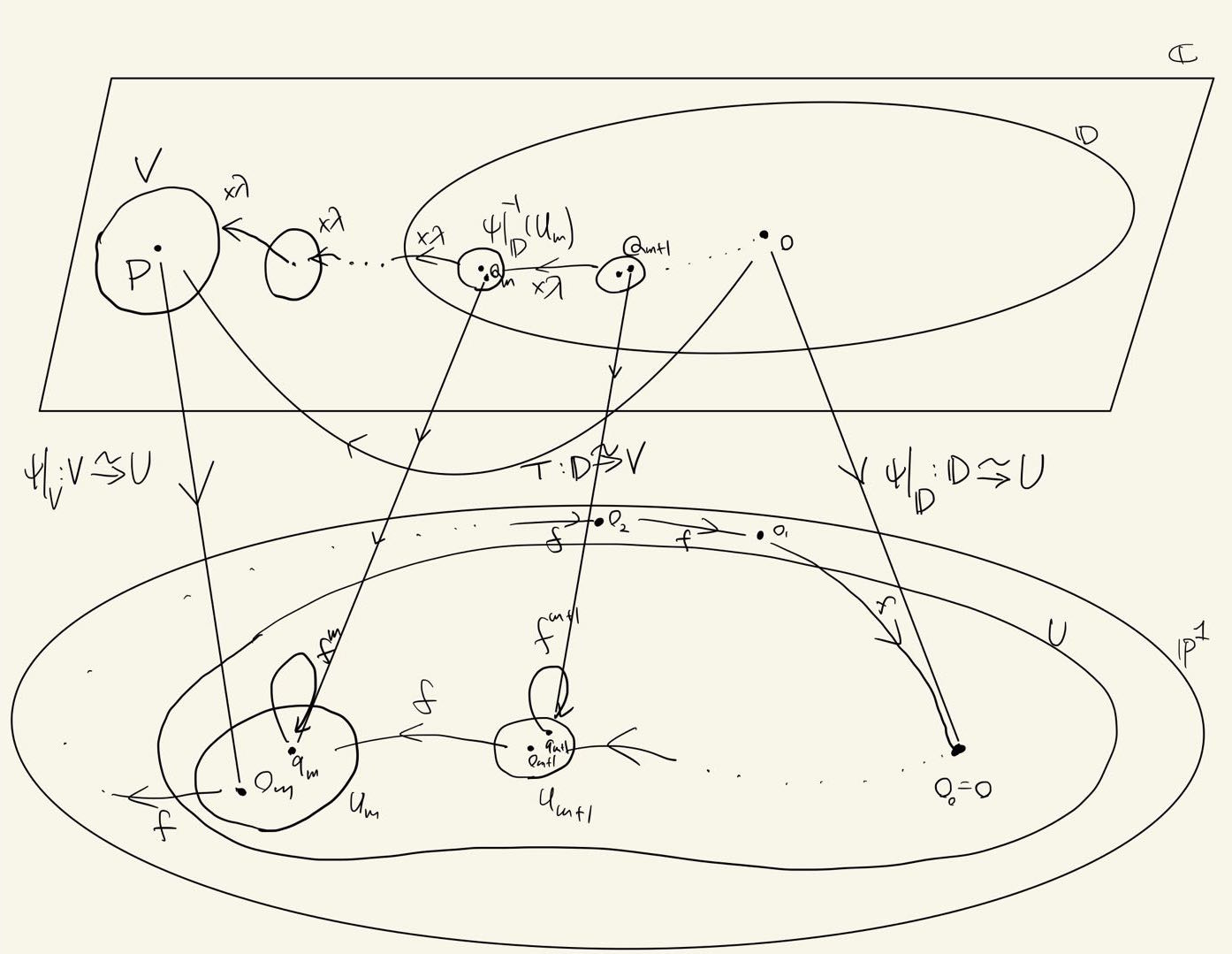}
\label{figc1}
\end{figure}

We have $\psi\circ T=\psi$ and $T(0)=P.$
Moreover, by our construction we have for every $i\geq m$, $V= \la^i (\psi|_{\D})^{-1}(U_i)$.  In particular $\la^i Q_i\in V$.  By (\ref{equationpoincare}) we have
\begin{equation*}
\psi(\la^i Q_i)=f^i(\psi(Q_i))=f^i(q_i)=q_i.
\end{equation*}
This implies

\begin{equation}\label{equationTQi}
	T(Q_i)=\la^iQ_i.
\end{equation}
Since $\lim\limits_{i\to \infty}q_i=o$, we have $\lim\limits_{i\to \infty}Q_i=0$ and 
\begin{equation}\label{equationlimQi} 
	\lim_{i\to \infty}\la^iQ_i=P.
\end{equation}
By (\ref{equationpoincare}) we have for every $i\geq 1$,
\begin{equation}\label{pocarederiva}
	df^i(\psi(z))d\psi(z)=\la^i d\psi(\la^i z),
\end{equation}

and by $\psi\circ T=\psi$ we have
\begin{equation}\label{Tderiva}
	d\psi(T(z))T'(z)=d\psi(z).
\end{equation}
Set $z=Q_i$. Combine (\ref{equationTQi}), (\ref{pocarederiva}) and (\ref{Tderiva}) we get 
\begin{equation*}
df^i(q_i)d\psi(\la^i Q_i)T'(Q_i)=\la^i d\psi(\la^i Q_i).
\end{equation*}
Since zeros of a holomorphic function are isolated, as $\la^i Q_i \to P$, for $i$ large enough we have $d\psi(\la^i Q_i)\neq 0$.  Hence for $i$ large enough,
\begin{equation}\label{equationdiffQi}
	\la^iT'(Q_i)^{-1}=df^i(q_i).
\end{equation}
\medskip

The following observation will be used in the proof of Theorem \ref{thmmilnor}.
\begin{lem}\label{lemmultilim}Set $\theta:=1/T': \D\to \C$. We have 
$$\lim_{i\to \infty}(df^i(q_i)-\la^i\theta(0))=P\theta'(0).$$
\end{lem}
\proof
By (\ref{equationlimQi}) and (\ref{equationdiffQi}),we have 
$$\lim_{i\to \infty}(df^i(q_i)-\la^i\theta(0))/P=\lim_{i\to \infty}(df^i(q_i)-\la^i\theta(0))/\la^iQ_i$$
$$=\lim_{i\to \infty}(df^i(q_i)/\la^i-\theta(0))/Q_i=\lim_{i\to \infty}(\theta(Q_i)-\theta(0))/Q_i=\theta'(0),$$
which concludes the proof.
\endproof

\medskip

The following is the main result of this section, which characterize exceptional endomorphisms by using the multipliers of adjoint sequence of a homoclinic orbit. 
\begin{thm}\label{thmadjconmultexcep}Let $f:\P^1\to \P^1$ be an endomorphism over $\C$ of degree at least $2$. Let $o$ be a repelling fixed point of $f$ such that $df(o)=\la$. Let $o_i, i\geq 0$ be a homoclinic orbit of $o$ such that $o_i\notin C(f)$ for every $i\geq 0$. Assume that there is $C\in \C^\ast$, such that for one (hence every) adjoint sequence $q_i, i\geq 0$ of $o_i, i\geq 0$, 
$df^i(q_i)=C\la^i$ for $i$ large. Then $f$ is exceptional. 
\end{thm}
\proof
We may assume that $q_i, i\geq m$ is adjoint with respect to the linearization domain and good return time $(U,m)$ for $o_i, i\geq 0$, and $d(f^i)(q_i)=C\la^i$ for all $i\geq m.$
By (\ref{equationdiffQi}), we get $T'(Q_i)=C^{-1}$ for $i\geq m.$ Since $Q_i\neq 0$ for $i\geq m$ and $\lim\limits_{i\to \infty}Q_i=0$,
$T'=C^{-1}$ on $\D$.  It follows that $T(z)=C^{-1}z+P$ for every $z\in \D.$ 
Then $T$ extends to the affine endomorphism on $\C$ sending  $z$ to $C^{-1}z+P.$ One have $\psi=\psi\circ T$ on $\C.$
We conclude the proof by Lemma \ref{lemaffineconjex}.
\endproof

\section{Proof of Milnor's conjecture}\label{sectionmilnor}

In this section we prove one of Milnor's conjectures (Theorem \ref{thmmilnor}). We postpone the proof of  another conjecture of Milnor (Theorem \ref{thmmilnor1}) to Section \ref{sectioncer}.

\proof[Proof of Theorem \ref{thmmilnor}]
Let $f:\P^1\to \P^1$ be an endomorphism over $\C$ of degree at least $2$. Let $K$ an imaginary quadratic field. 
Assume that for every $n\geq 1$ and every $n$-periodic point $x$ of $f$, $df^n(x)\in O_K.$ 

\medskip

After replacing $f$ by a suitable positive iterate, we may assume that $f$ has a repelling fixed point $o\notin PC(f)$.
Let $o_i, i\geq 0$ be a homoclinic orbit of $o.$ 
By Proposition \ref{proexihors}, there is a linearization domain and a good return time $(U,m)$ for $o_i, i\geq 0$. Let $q_i, i\geq m$ be the adjoint sequence for it.
Set $\mu_i:=df^i(q_i)\in O_K$ for $i\geq m$.  Set $\la:=df(o).$
\begin{lem}\label{lemmuiab}There are $a\in K^*,b\in K$ such that $\mu_i=a\la^i+b$ for $i$ large.
\end{lem}
\proof[Proof of Lemma \ref{lemmuiab}]
We view $K$ as a subfield of $\C.$ Then $O_K$ is a discrete subgroup of $(\C, +)$.
Set $\T:=\C/O_K$ and $\pi: \C\to \T$ the quotient map. Since $\la\in O_K$, the multiplication by $\la$ on $L$ descents to an endomorphism $[\la]$ on $\T.$ 
By Lemma \ref{lemmultilim}, we have 
\begin{equation}\label{equationlimitmui}\lim_{i\to \infty}(\mu_i-a\la^i)=b,
\end{equation} where 
$a=\theta(0)=1/T'(0)\in \C^*$ and $b=P\theta'(0)\in \C$ (See Section \ref{sectionhomoclinic} for the definitions of $T$ and $\theta$). Since $\mu_i\in O_K, i\geq m$, we get 
$$\lim_{i\to \infty}[\la]^i\pi(a)=\pi(b).$$
In particular, $\pi(b)$ is fixed by $[\la].$ Since $d[\la](b)=\la$, $[\la]$ is repelling at $\pi(b).$ Hence for $i$ large we must have
\begin{equation}\label{equationeqmui}
	[\la]^i\pi(a)=\pi(b).
\end{equation}
Since $O_K$ is discrete in $\C$, by 
(\ref{equationlimitmui}) and (\ref{equationeqmui}),  we have
\begin{equation}\label{equationabequmu}\mu_i=a\la^i+b \text{ for $i$ large}.
\end{equation}
There are $n> l\geq m$ such that 
$\mu_n=a\la^n+b$ and $\mu_l=a\la^l+b$.  This implies that $a,b\in K.$ 
\endproof

After enlarging $m,$ we may assume that $\mu_i=a\la^i+b$ for all $i\geq m.$ Assume by contradiction that $f$ is not exceptional. 
By Theorem \ref{thmadjconmultexcep}, we must have $b\neq 0.$ 
For $\bp\in \Spec O_K$, 
let $K_\bp$ 
be the completion of $K$ with respect to $\bp$. Denote by $|\cdot|_\bp$ the $\bp$-adic norm on $K_\bp$ normalized by $|p|_{\bp}=p^{-1}$ where $p:=\Char O_K/\bp.$
Let $K_\bp^{\circ}$ be the valuation ring of $K_{\bp}$. 
For $\mu\in O_K$, $\mu\in \bp$ if and only if $|\mu|_\bp<1$.
\begin{lem}\label{lemoneinfimuip}For $\bp\in \Spec O_K$ and $\epsilon>0$, if $\la\not\in \bp$, then
	there is $N\in \Z_{>0}$ such that $|\la^{Ni}-1|_\bp<\epsilon$ for all $i\geq 0.$
\end{lem}
\proof[Proof of Lemma \ref{lemoneinfimuip}] 
Since $O_K/\bp$ is a finite field and $\la\not\in \bp$, there is $l\geq 1$ such that $\la^l-1\in \bp.$ Since
$$\lim_{n\to \infty}\la^{lp^n}=\lim_{n\to \infty}(1+(\la^l-1))^{p^n}=1$$ in the $\bp-$adic topology,
there is $N\in \Z_{>0}$, such that $|\la^N-1|_{\bp}<\epsilon.$ Then for every $i\geq 0$,
$|\la^{Ni}-1|_{\bp}=|\la^{N}-1|_{\bp} |1+\la^N\dots+\la^{N(i-1)}|_{\bp}<\epsilon.$
\endproof

Let $S$ be the finite set of prime ideals $\bp\in \Spec O_K\setminus \{0\}$ dividing $\la(\deg f)!\in O_K.$
For every $\bp\in \Spec O_K\setminus(S\cup \{0\})$, there is an embedding of field $\tau_K: K\hookrightarrow \C_p$ such that $|\cdot|_{\bp}$ is the restriction of the norm on  $\C_p$ via this embedding. Recall that  $\C_p$ is the completion of the algebraic closure of $\Q_p$. Then $\tau_K$ extends to an isomorphism $\tau:\C\to \C_p$. Via $\tau$, the norm $|\cdot|_{\bp}$  extends to a non-archimedean complete norm on $\C$. 
By \cite[Corollaire 4.7 and Corollaire 4.9]{Rivera-Letelier2003a} of Rivera-Letelier (or \cite[Corollary 1.6]{Benedetto2014} of Benedetto-Ingram-Jones-Levy), for every $\bp\in \Spec O_K\setminus(S\cup \{0\})$, there are at most finitely many integers $i\geq m$ satisfying $|\mu_i|_\bp<1$. We claim that  for every $i\geq m,$ we have $\mu_i=a\la^i+b\not\in \bp$ for every $\bp\in \Spec O_K\setminus(S\cup \{0\})$. 
In fact if there is $\bp\in \Spec O_K\setminus (S\cup \{0\})$ such that $a\la^i+b\in \bp$ for some $i\geq m$, by Lemma \ref{lemoneinfimuip}, there is $N\in \Z_{>0}$, such that for all 
$j\geq 0$, $|\la^{Nj}-1|_\bp<|a^{-1}|/2.$ Then for every $j\geq m$, we get 
$$|\mu_{i+Nj}|_{\bp}=|a\la^{i+Nj}+b|_{\bp}\leq \max\{|a\la^{i}+b|_{\bp}+|a\la^i(\la^{Nj}-1)|_{\bp}\}<1.$$
Thus we obtain infinitely many  integers $i\geq m$ satisfying $|\mu_i|_\bp<1,$ which is a contradiction.

\medskip
Set $S':=\{\bp\in S|\,\, \la\in \bp\}$ and $S''=S\setminus S'.$
Since $a\neq 0$, there is $l\geq 0$ such that $a\la^l+b\neq 0$. Set $$A:=\min(\{|a\la^{l}+b|_{\bp}|\,\, \bp\in S''\}\cup \{|b|_{\bp}|\,\, \bp\in S'\}\cup \{1\})>0.$$

For every $\bp\in S'$, there is an integer $M_{\bp}\geq m$ such that 
$|a\la^{M_{\bp}}|_{\bp}<|b|_{\bp}.$ Then, for every $i\geq M_{\bp}, \bp\in S'$, we have $$|\mu_i|_{\bp}=|b|_{\bp}\geq A.$$

For every $\bp\in S''$, by Lemma \ref{lemoneinfimuip}, there is $N_{\bp}\in \Z_{>0}$ such that for every $j\geq 0$,
$|\la^{N_{\bp}j}-1|_{\bp}<|a^{-1}|_{\bp}|a\la^{l}+b|_{\bp}.$ Then for all $j\geq m$,  we have
$$|\mu_{l+N_{\bp}j}|_{\bp}=|a\la^{l+N_{\bp}j}+b|_{\bp}=|(a\la^{l}+b)+a\la^i(\la^{N_{\bp}j}-1)|_{\bp}=|a\la^{l}+b|_{\bp}\geq A.$$

Set $M:=\max\{M_{\bp}|\,\, \bp\in S'\}$ and $N:=\prod_{\bp\in S''}N_{\bp}.$ For every $i\geq M$, by the above discussion we get 
$|\mu_{l+Ni}|_{\bp}\geq A$ for all $\bp\in S$. Fix an embedding of $K$ in $\C$.
For every $\bp\in \Spec O_K\setminus \{0\},$ set $n_{\bp}:=[K_{\bp}:\Q_{p}]$ with $p=\Char O_K/\bp.$ We have $n_\bp\leq 2$.
By product formula, we get, since $|\mu_{l+Ni}|_{\bp}=1$ for all $\bp\in \Spec O_K\setminus(S\cup \{0\})$,
$$|\mu_{l+Ni}|^{[K:\Q]}=\prod_{\bp\in \Spec O_K\setminus \{0\}}|\mu_{l+Ni}|_{\bp}^{-n_{\bp}}=\prod_{\bp\in S}|\mu_{l+Ni}|_{\bp}^{-n_{\bp}}\leq A^{-2|S|},$$ where $i\geq m$.

Hence $\mu_{l+Ni}, i\geq m$ is bounded in $\C.$ Since $a\neq 0$ and $|\la|>1$, we get a contradiction.  The proof is finished.
\endproof

\section{The Berkovich projective line}\label{sectionberko}
Let $\bk$ be a complete valued field with a non-trivial non-archimedean norm $|\cdot|$. We denote by $\bk^{\circ}$ the valuation ring of $\bk$, $\bk^{\circ\circ}$ the maximal ideal of $\bk^{\circ}$ and $\tilde{k}=\bk^{\circ}/\bk^{\circ\circ}$ the residue field.

In this section, we collect some basic facts about Berkovich's analytification of $\P^1_{\bk}$. 
We refer the readers  to \cite{Berkovich1990} for a general discussion on Berkovich space, and to \cite{baker2010potential} for a detailed description of the Berkovich projective line and the dynamics on it.

\subsection{Analytification of the projective line}
Let $\P^{1,\an}_{\bk}$ be the analytification of $\P^{1}_{\bk}$ in the sense  of Berkovich, which is a compact topological
space endowed with a structural sheaf of analytic functions. 
Only its topological structure will be used in this paper. We describe it briefly below.

The analytification $\A^{1,\an}_{\bk}$ of the affine line $\A^1_{\bk}$
is the space of all multiplicative semi-norms on $\bk[z]$ whose restriction to $\bk$ coincide with $|\cdot|$, endowed with the topology of pointwise convergence.
For any $x\in \A^{1,\an}_{\bk}$ and $P\in \bk[z]$, it is customary to denote $|P(x)|:=|P|_x$, where $|\cdot|_x$ is the semi-norm associated to $x$.

As a topological space $\P^{1,\an}_{\bk}$ is the one-point compactification of $\A_{\bk}^{1,\an}$. We write $\P_{\bk}^{1,\an}=\A_{\bk}^{1,\an}\cup \{\infty\}$. More formally it is obtained
by gluing two copies of $\A_{\bk}^{1,\an}$ in the usual way via the transition map $z\mapsto z^{-1}$ on the 
punctured affine line $(\A_{\bk}^1\setminus \{0\})^{\an}$.

The Berkovich projective line $\P^{1,\an}_{\bk}$ is an $\R$-tree in the sense that it is uniquely path
connected, see \cite[Section 2]{Jonsson} for the precise definitions. In particular for $x,y\in \P^{1,\an}_{\bk}$,
there is a well-defined segment $[x,y]$. 

For $a\in \bk$ and $r\in [0,+\infty)$, we denote  $\D(a,r)$ by the closed disk $\D(a,r):=\{x\in \A^{1,\an}_{\bk}:\, |(z-a)(x)|\leq r\}.$
One may check that the norm $\sum_{i\geq 0}a_i(z-a)^i\mapsto \max\{|a_i|r^i, i\geq 0\}$ defines a point $\xi_{a,r}\in \D(a,r).$ 
One set $x_{G}:=\xi_{0,1}$ and call it the Gauss point.
\begin{rem}When $r=0$, $\xi_{a,0}$ is exactly the image of $a$ via the identification $\bk=\A^1(\bk)\hookrightarrow \A_{\bk}^{1,\an}.$
\end{rem}

The group $\PGL_2(\bk)$ acts on $\P^1_{\bk}$, hence on $\P^{1, \an}_{\bk}$. 
\begin{lem}\cite[Proposition 1.4]{Dujardin2019}\label{lemgaussorbit}
	For a point $x\in \P^{1, \an}_{\bk}$, $x\in \PGL_2(\bk)\cdot x_g$ if and only if it takes form 
$x=\xi_{a,r}$ for some $a\in \bk$ and $r\in |\bk^*|.$
\end{lem}

\begin{rem}\label{remdensesubf}
The stablizer of $\PGL_2(\bk)$ at $x_g$ is $\PGL_2(\bk^{\circ})$ which is open in $\PGL_2(\bk)$. So for any dense subfield $L$ of $\bk$, we have $\PGL_2(L)\cdot x_g=\PGL_2(\bk)\cdot x_g.$
\end{rem}

\subsection{Points in $\P^{1,\an}_{\bk}$}
Let $\widehat{\overline{\bk}}$ be the completion of the algebraic closure of $\bk$. It is still algebraically closed. 
By \cite[Corollary 1.3.6]{Berkovich1990}, $\Aut(\widehat{\overline{\bk}}/\bk)$ acts on $\P^{1,\an}_{\widehat{\overline{\bk}}}$ and we have $\P^1_{\widehat{\overline{\bk}}}/\Aut(\widehat{\overline{\bk}}/\bk)= \P^1_{\bk}.$ We denote by $\pi:\P^{1,\an}_{\widehat{\overline{\bk}}}\to \P^{1,\an}_{\bk}$ the quotient map. 
The points of $\P^{1,\an}_{\bk}$  can be classified into 4 types:
\begin{points}
\item  a type 1 point takes form $\pi(a)$ where $a\in \widehat{\overline{\bk}}\cup\{\infty\}=\P^{1,\an}_{\widehat{\overline{\bk}}}$;
\item a type 2 point takes form $\pi(\xi_{x,r})$ where $x\in \widehat{\overline{\bk}}$ and $r\in |\widehat{\overline{\bk}}^*|$;
\item a type 3 point takes form $\pi(\xi_{x,r})$ where $x\in \widehat{\overline{\bk}}$ and $r\in \R_{>0}\setminus |\hat{\overline{\bk}}^*|$;
\item a type 4 point takes form $\pi(x)$ where $x$ is the pointwise limit of $\xi_{x_i,r_i}$ such that the corresponding discs $\D(x_i,r_i)$ form a decreasing
sequence with empty intersection.
\end{points}
See \cite[Section 1.4.4]{Berkovich1990} for further details when $\bk$ is algebraically closed. See also \cite[Proposition 2.2.7]{Kedlaya2011} and \cite[Section 2.1]{Stevenson2019}.  The set of type 1 (resp. type 2) points is dense in $\P^{1,\an}_{\bk}$. Points of type 4 exist only when
$\bk$ is not spherically complete. If we view $\P^{1,\an}_{\bk}$ as a metric tree, then the end points have type 1 or 4.

\medskip

For every $x\in \P^{1,\an}_{\bk}$, we can define an equivalence relation on the set $\P^{1,\an}_{\bk}\setminus\{x\}$ as follows: $y\sim z$ if the paths $(x, y]$ and
$(x,z]$ intersect. The tangent space $T_{x}$ at $x$ is the set of equivalences classes of $\P^{1,\an}_{\bk}\setminus\{x\}$ modulo $\sim$.  
See \cite[Section 2.5]{Jonsson}
for details. If $x$ is an end point (a point of type 1 or 4), then $|T_{x}|=1$. If  $x$ is of type 3, then $|T_{x}|=2$. If $x$ is of type 2, then $|T_{x}|\geq 3.$
For a direction $v\in T_x$, let $U(v)$ be the set of all $y\in \P^{1,\an}_{\bk}$ such that the path $(x,y]$ presents $v.$ Then $U(v)$ is an open subset such that $\partial{U(v)}=x$.

\subsection{Dynamics on $\P_{\bk}^{1,\an}$} Let $f:\P^1_{\bk}\to\P^1_{\bk}$ be an endomorphism of degree $d\geq 2.$
We still denote by $f$ the induced endomorphism on $\P^{1,\an}_{\bk}$. 

\subsubsection{The tangent map}
For $x,y\in \P_{\bk}^{1,\an}$, if $f(x)=y$, then $x,y$ have the same type. Moreover, $f$ induces a tangent map $T_xf: T_x\to T_y$ sending
$v\in T_x$ to the unique direction $w\in T_y$ such that for every $z\in U(v)$, $(y,f(z)]\cap U(w)\neq\emptyset.$
We note that, in general $f(U(v))$ may not equal to $U(w).$ If $f(U(v))=U(w)$, we say that $v$ is a \emph{good direction}. Otherwise, it is called a \emph{bad direction}. 
If $v$ a bad direction, then $f(U(v))=\P^{1,\an}_{\bk}$  \cite[Theorem 7.34]{Benedetto2019a}.

\medskip

We may naturally identify $T_{x_{G}}$ with $\P^1(\tilde{\bk})$ as follows: Consider the standard model $\P^1_{\bk^{\circ}}$ of $\P_{\bk}^{1,\an}$. There is a reduction map $\red: \P_{\bk}^{1,\an}\to \P^1_{\tilde{\bk}}.$ The preimage of the generic point of $\P^1_{\tilde{\bk}}$ is the Gauss point $x_G$ and for every $y\in \P^1(\tilde{\bk})$, there is a unique $v_y\in T_{x_{G}}$ such that $U(v_y)=\red^{-1}(y).$ The map $\P^1(\tilde{\bk})\to T_{x_{G}}$ sending $y$ to $v_y$ is bijective. Let $h$ be any endomorphism of $\P^1_{\bk}$ such that $h(x_G)=x_G$, it extends to a rational self-map $h_{\bk^{\circ}}$ of $\P^1_{\bk^{\circ}}$. We denote by $\tilde{h}: \P^1_{\tilde{\bk}}\to \P^1_{\tilde{\bk}}$ the restriction of $h$ to the special fiber of $\P^1_{\bk^{\circ}}$ and call it the reduction of $h.$ Then $T_{x_G}h: T_{x_G}=\P^1(\tilde{\bk}) \to T_{x_G}$ is induced by $\tilde{h}$. We define $\deg T_{x_G}h$ to be the degree of $\tilde{h}.$
We note that $\deg \tilde{h}\leq \deg h$. The equality holds if and only if $h_{\bk^{\circ}}$ is an endomorphism. In this case, we say that $h$ has \emph{explicit good reduction}.

\medskip

More generally, for every $x,y\in \PGL_2(\bk)\cdot x_{G}$ with $f(x)=y$, we may define $$\deg T_{x}f:=\deg T_{x_G}(h^{-1}fg)=\deg \widetilde{h^{-1}fg},$$ where $h,g\in \PGL_2(\bk)$ with $g(x_G)=x$ and $h(x_G)=y.$ Then $1\leq \deg T_{x_G}f\leq \deg f$ and $\deg T_{x_G}f$ does not depend on the choices of $g,h.$
\rem\label{remalgcloseorbittypetwo}Assume that $\bk$ is algebraically closed. By Lemma \ref{lemgaussorbit}, the set of type 2 points in $\P^{1,\an}_{\bk}$ is exactly $\PGL_2(\bk)\cdot x_G.$
\endrem

\subsubsection{Periodic points}
Assume that $\bk$ is algebraically closed.
For $n\geq 1$, a $n$-periodic point of $f$ is a point $x\in \P^{1,\an}_{\bk}$ such that $f^n(x)=x.$ 
They can be divided into three types: attracting, indifferent and repelling.
A type 1 periodic point $x\in \P^{1}(\bk)$ of period $n\geq 1$ is called \emph{attracting} if $|d(f^n)(x)|<1$; \emph{indifferent} if $|d(f^n)(x)|=1$; and  \emph{repelling} if $|d(f^n)(x)|>1.$
A $n$-periodic point $x\in \P_{\bk}^{1,\an}$ of type 2 is called \emph{indifferent} if $\deg T_{x}f=1$; \emph{repelling} if $\deg T_{x}f\geq 2.$
Every $n$-periodic point $x\in \P_{\bk}^{1,\an}$ of type 3 or 4 are  \emph{indifferent}  \cite[Lemma 5.3, 5.4]{Rivera-Letelier2003}.

\subsubsection{Fatou and Julia sets}Assume that $\bk$ is algebraically closed.

\medskip

The \emph{Julia set} of $f$ is the set $\sJ(f)$ of points $z \in \P^{1,\an}_{\bk}$ with the following property: for every neighborhood $U$ of $z$, the union of iterates $\bigcup _{n\geq 0} f^n(U)$ omits only finitely many points of $\P^{1,\an}_{\bk}$. Its complement $\sF(f):=\P^{1,\an}_{\bk}\setminus \sJ(f)$ is the \emph{Fatou set} of $f$.

\medskip

We list some basic properties of the Julia and Fatou sets of $f$.
\begin{pro}\cite[Chapter 8 and Section 12.2]{Benedetto2019a}\label{probasicfj}
\begin{points}
\item The Fatou set $\sF(f)$ is open and the Julia set $\sJ(f)$ is closed.

\item All attracting periodic points of $f$ are contained in $\sF(f)$. 

\item All repelling periodic points of $f$ are contained in $\sJ(f)$.

\item We have $\sJ(f)=f(\sJ(f))=f^{-1}(\sJ(f))$ and  $\sF(f)=f(\sF(f))=f^{-1}(\sF(f))$.

\item Both  $\sJ(f)$ and $\sF(f)$ are nonempty.

\item For every  $z\in \sJ(f)$, $\cup_{n\geq 0}f^{-n}(z)$ is dense in $\sJ(f)$.

\item Repelling periodic points are dense in  $\sJ(f)$.
\end{points}
\end{pro}
\subsubsection{Good reduction}
We say $f$ has \emph{good reduction} if after some coordinate change $h \in \PGL_2(\bk)$, the map $h^{-1} \circ f \circ h $ has explicit good reduction.
\begin{thm}\cite[Theorem E]{Favre2010}\label{thmgoodredjulia}
The endomorphism $f$ has explicit good reduction if and only if $\sJ(f)=x_G$.
Moreover, if $\bk$ is algebraically closed, 
$f$ has good reduction if and only if $\sJ(f)$ is a single point.
\end{thm}

\rem\label{remjuliaonetypet}Assume that $\bk$ is algebraically closed. If $\sJ(f)$ is a single point, then by Theorem \ref{thmgoodredjulia} and (vii) of Proposition \ref{probasicfj}, it is a type 2 repelling point. 
\endrem


\section{Rescaling limits of holomorphic families}\label{sectionrescaling}
\subsection{Holomorphic families}

Recall that $\Psi: \text{Rat}_d(\C)\to \sM_d(\C)$ is the quotient morphism, where $\sM_d(\C):=\text{Rat}_d(\C)/\PGL_2(\C)$ is the moduli space.

Let $\La$ be a complex manifold, we denote by $\sO^{\an}(\La)$  the ring of holomorphic functions on $\La.$ 
 Moreover, if $\La$ is complex algebraic variety, we denote by $\sO(\La)$ the ring of algebraic functions on $\La.$

 A \emph{holomorphic (resp. meromorphic) family} on $\Lambda$ is an endomorphism (resp. meromorphic self-map) $f_{\La}$ on $\P^1\times \La$ such that $\pi_{\La}\circ f_{\La}=\pi_{\La}$, where $\pi_{\La}:\P^1(\C)\times \La\to \La$ is the projection to $\La.$   
 More concretely, one may write $f_{\La}([x:y],t)=([P_t(x,y):Q_t(x,y)], t)$  where $P_t(x,y), Q_t(x,y)$ are homogenous polynomials of same degree $d$ in $\sO^{\an}(\La)[x,y]$ without common divisor. We say that $f_{\La}$ is of degree $d$.
 Then $f_{\La}$ is holomorphic if there is no $(t,x,y)\in \La\times \C^*\times\C^*$ such that $P_t(x,y)=Q_t(x,y)=0.$

For $t\in \La$, we denote by $f_t$ the restriction of $f_{\La}$ to the fiber above $t.$ We denote by $I(f_{\La})$ the indeterminacy locus of $f_{\La}$ and $B(f_{\La}):=\pi_{\La}(I(f_{\La})).$ Then $I(f_{\La})$ and $B(f_{\La})$ are proper closed analytic subspace of $\P^1\times \La$ and $\La$ respectively.  For every $t\in \La\setminus B(f_{\La})$, we have $\deg f_t=d.$ When $\La$ is connected, this is equivalent to say that  $\deg f_t=d$ for one $t\in \La\setminus B(f_{\La}).$ A meromorphic family is holomorphic if and only if $B(f_{\La})=\emptyset$. 
  
 So give a degree $d$ holomorphic family $f_{\La}$ on $\La$ is equivalent to give a holomorphic morphism $t\mapsto f_t=P_t/Q_t$ from $\La$ to $\text{Rad}_d(\C).$
We say that $f_{\La}$ is algebraic if $\La$ is a complex algebraic variety and $f_{\La}: \P^1\times \La\to \P^1\times \La$ is algebraic i.e. $P_t,Q_t \in \sO(\La)[x,y].$
In other words, it means that the induced morphism $\La\to \text{Rad}_d(\C)$ is algebraic.

\medskip 

For a degree $d$ holomorphic family $f_{\La}$ on $\La$, let $\Psi_{\La}: \La\to \sM_d$ be the holomorphic morphism sending $t\in \La$ to the class of $f_t$ in $\sM_d(\C).$
 We say that $f_{\La}$ is \emph{isotrivial} if $\Psi_{\La}: \La\to \sM_d$ is locally constant. More generally for degree $d$ meromorphic family $f_{\La}$, we say that $f_{\La}$ is \emph{isotrivial} if 
 $f|_{\La\setminus B(f_{\La})}$ is isotrivial.
 \medskip

 \subsection{Potentially good reduction}   
 Assume that $\La$ is a Riemann surface and $f_{\La}$ is a meromorphic family of degree $d.$

 For $b\in \La$, we say that $f_{\La}$ has \emph{potentially good reduction} at $b$ if $\Phi_{\La\setminus (B(f_{\La})\cup\{b\})}: \La\to \sM_d$ extends to a holomorphic morphism 
 on $(\La\setminus B(f_{\La}))\cup \{b\}.$  In particular, $f_{\La}$ has potentially good reduction at every $b\in \La\setminus B(f_{\La}).$

 \begin{lem}\label{lemeverygoodisotrivial}Assume that $\La$ is an irreducible smooth projective curve. Let $f_{\La}$ be a meromorphic family of degree $d.$ 
 If $f_{\La}$ has potentially good reduction at every point in $\La$, then $f_{\La}$ is isotrivial. 
 \end{lem}
 
 \proof
 Since  $f_{\La}$ has potentially good reduction at every point in $B(f_{\La})$, $\Psi_{\La\setminus B(f_{\La})}: \La\setminus B(f_{\La})\to \sM_d$ extends to a holomorphic morphism 
 $\Psi_{\La}:\La\to \sM_d.$ 
Recall that $\sM_d(\C)=\Spec (\sO(\text{Rat}_d(\C)))^{\PGL_2(\C)}$ is affine \cite[Theorem 4.36(c)]{Silverman2007}.
This follows from the fact that $\text{Rat}_d(\C)$ is affine and the geometric invariant theory \cite[Chapter 1]{Mumford1982}. 
 Since $\La$ is projective, $\Psi_{\La}$ is a constant map. This concludes the proof.
 \endproof

Having potentially good reduction is a local property at $b$, i.e.  for every open neighborhood $U$ of $b$ in $\La$, $f_{\La}$ has potentially good reduction at $b$ if and only if $f_{U}:=f_{\La}|_{\P^1(\C)\times U}$ has potentially good reduction at $b$.  Note that there is an open neighborhood $U$  of $b$ which is isomorphic to a disk $\D$ such that $f_{U\setminus \{b\}}$ is holomorphic.  
So we can focus on the case that $f_{\D}$ is a meromorphic family which is holomorphic on $\D^*.$ 
We will give another characterization of potentially good reduction via non-archimedean dynamics.

\subsection{Holomorphic family on puncture disk}\label{subsectionholpunc}
Let $f_{\D}$ be a a meromorphic family of degree $d\geq 2$ which is holomorphic on $\D^*.$ Let $t$ be the standard coordinate on $\D$.
We can relate $f_{\D}$ to some non-archimedean dynamics on the field of Laurent's series $\C((t)).$

Recall that on  $\C((t))$, there is a $t$-adic norm $|\cdot|$: Given an element $z=\sum_{n\geq n_0} a_n t^n\neq 0$, where $n_0\in\Z$, $a_n\in \C$ and $a_{n_0}\neq 0$, the $t$-adic norm of $z$ is $|z|:=e^{-n_0}$. This norm is non-archimedean and $\C((t))$ is complete for $|\cdot|$.
 Set $\LL:=\widehat{\overline{\C((t))}}.$ 

\medskip

\par Write $$f([x:y],t)=([P_t(x,y):Q_t(x,y)], t)$$ 
where $P_t(x,y), Q_t(x,y)$ are homogenous polynomials of degree $d$ in $\sO^{\an}(\D)[1/t][x,y]$ without common divisors. 
 Since $\sO^{\an}(\D)[1/t]\subseteq \C((t))$, $f_{\D}$ defines an endomorphism $f_{\C((t))}: [x,y]\mapsto [P_t(x,y):Q_t(x,y)]$ on $\P^1_{\C((t))}$ of degree $d.$ 
  Set $f_{\LL}:=f_{\C((t))}\hotimes_{\C((t))}\LL:\P^1_\LL \to \P^1_\LL$. 
  
  \medskip

 Recall that 
 \begin{equation}\label{equationbci}\overline{\C((t))}=\cup_{n\geq 1}\C((t^{1/n})).
 \end{equation}
To get endomorphisms over $\C((t^{1/n}))$, we introduce some base changes of $f_{\D}$ as follows.
Consider the morphism $\phi_n: U_n:=\D\to \D$ sending $t$ to $t^n.$ There is $u_n\in \sO^{\an}(U_n)$ such that $u_n^n=\phi^*t.$ Then $u_n$ is a coordinate on $U_n$ and we may identify 
 $\C[u_n]$ with $\C[t^{1/n}]$ (hence we may identify  $\C((u_n))$ with $\C((t^{1/n}))$). Let $o\in U_n$ be the point defined by $u_n=0.$ 
 The endomorphism  on $\P^{1,\an}_{\C((u_n))}$ induced by $f_{U_n}$ is $f_{\C((u_n))}=f_{\C((t))}\hotimes_{\C((t))}\C((t^{1/n})).$

 \begin{lem}\label{lempotgoodr}If $f_{\LL}$ has good reduction then $f_{\D}$ has potentially good reduction at $0$.
 \end{lem}
 
  \begin{rem}
  The inverse statement of Lemma \ref{lempotgoodr} is also true. However, we do not need that direction in this paper. So we leave it to readers.
 \end{rem}

 \proof[Proof of Lemma \ref{lempotgoodr}] 
 By Theorem \ref{thmgoodredjulia}, there is $h\in \PGL_2(\LL)$ such that $\sJ(f_{\LL})=\{h(x_G)\}.$
 Then  $h^{-1}\circ f_{\LL}\circ h$ has explicit good reduction. 
 By (\ref{equationbci}) and Remark \ref{remdensesubf}, we may assume that $h\in \PGL_2(\C((t^{1/n})))$ for some $n\geq 1.$
 Since $\C(u_n)$ is dense in $\C((u_n))=\C((t^{1/n}))$, by Remark \ref{remdensesubf} again, we may assume that $h\in \PGL_2(\C(u_n)).$ There is an open neighborhood $V$ of $o$ such that 
 $h$ and $h^{-1}$ are holomorphic on $V\setminus \{o\}$ i.e. they define holomorphic families $h_{V\setminus \{o\}}$ and $h^{-1}_{V\setminus \{o\}}.$ We may assume further that $V\simeq \D.$
 Consider the family $f_V:=h_V^{-1}\circ f_{U_n}|_V\circ h_V$. 
 Observe that 
 \begin{equation}\label{equationpsiv}\Psi_{\D^*}\circ \phi|_{V\setminus \{o\}}=\Psi_{V\setminus\{o\}}.
 \end{equation}
 Then $f_V$ induces an endomorphism $f_{\C((u))}=f_{\C((t))}\hotimes_{\C((t))}\C((u))$
 on $\P^{1,\an}_{\C((u))}$, which has good reduction. So $f_V$ is an endomorphism on $\P^1\times V$. So $\Psi_{V\setminus\{o\}}$ extends to a holomorphic morphism $\Psi_V: V\to \sM_d.$ By (\ref{equationpsiv}), $\Psi_{\D^*}$ is bounded in some neighborhood of $o.$ So $\Psi_{\D^*}$ extends to a holomorphic morphism on $\D$, which means that $f_{\D}$ has potentially good reduction at $0$.
\endproof


%


The following definition was introduced by Kiwi.
\begin{defi}\cite{kiwi2015rescaling}\label{defiresclim}
	Let $f_{\D}$ be a meromorphic family of degree $d\geq 2$ which is holomorphic on $\D^*.$ We say an endomorphism $g$ is a rescaling limit 
	of $f_{\D}$ (or $f_{\D^*}$) (via $(q,M_{\D})$) if there is an integer $q\geq 1$, a finite set $S\subset \P^1(\C)$ and a meromorphic family $M_{\D}$ of degree 1, such that $M_{\D}$ and $M^{-1}_{\D}$ are holomorphic on $\D^*$ and
	\begin{equation*}
		M_t^{-1}\circ f_t^q\circ M_t(z)\to g(z)
	\end{equation*}
when $t\to 0$ , uniformly on compact subsets of $\P^1(\C)\setminus S$.
	\end{defi}

The following result was proved by Kiwi.
\begin{pro}\cite[Proposition 3.4]{kiwi2015rescaling}\label{kiwirescaling}
Let $f_{\D}$ be a meromorphic family of degree $d\geq 2$ which is holomorphic on $\D^*$.  Let  $M_{\D}$  be a meromorphic family of degree $1$ , such that $M_{\D}$ and $M^{-1}_{\D}$ are holomorphic on $\D^*$. 
Then for all $q\geq 1$, the following are equivalent:

\begin{points}
\item There exist an endomorphism $g$ on $\P^1$  and a finite set $S\subset \P^1(\C)$  satisfy
	\begin{equation*}
	M_t^{-1}\circ f_t^q\circ M_t(z)\to g(z)
\end{equation*}
when $t\to 0$ , uniformly on compact subsets of $\P^1(\C)\setminus S$.

\item The point $x=M_{\LL}(x_G)$ is fixed by $f_{\LL}^q$ and $\widetilde{M_{\LL}^{-1}\circ f_{\LL}^q\circ M_{\LL}}=g.$
\end{points}

In the case (i) and (ii) hold,  $T_x f^q: T_x \to T_x $ can be identified with $g$ after identify $T_x $ to $T_{x_G}=\P^1(\C)$ via $T_{x_G}M_{\LL}: T_{x_G}\to T_{x}$. 
Under this identification, $S$ is a finite subset of $T_x$ which contains all  the bad directions of $T_x f^q$. 
\end{pro}

\rem\label{remresgeo} 
One may rewrite Definition \ref{defiresclim} in the following more geometric way:
Let $h_{\D}$ be the meromorphic family $h_\D:=M_\D^{-1}\circ f^q_{\D}\circ M_\D$ on $\P^1(\C)\times \D$, then $h_0=g$.
Moreover,  $S$ can be any finite subset containing $S_0$ where $I(h_{\D})=S_0\times \{0\}\subseteq\P^1(\C)\times \D$.
\endrem

\begin{cor}\label{cortypetwomodel}Let $x\in \P_{\LL}^{1,\an}$ be a type 2  fixed point of $f_{\LL}.$ Assume that $T_xf_{\LL}$ is conjugate to some endomorphism $g: \P^1(\C)\to \P^1(\C).$ Then there is $n\geq 1$, such that $g$ is a rescaling limit of $f_{U_n}|_V$ where $f_{U_n}$ is the base change of $f_{\D}$ by the morphism $U_n:=\D\to \D$ sending $t$ to $t^n$ as in Section \ref{subsectionholpunc} and $V$ is an open neighborhood of $o\in U_{n}$ isomorphic to $\D.$
\end{cor}

\proof
There is $M_{\LL}\in \PGL_2(\LL)$ such that $x=M_{\LL}(x_G)$,
 by (\ref{equationbci}) and Remark \ref{remdensesubf}, we may assume that $M_{\LL}\in\PGL_2(\C((t_n^{1/n})))$ for some $n\geq 1.$ 
Let $f_{U_n}$ be the base change of $f_{\D}$ by the morphism $\phi_n:U_n:=\D\to \D$ sending $t$ to $t^n$ 
and pick $u_n$ with $u_n^n=\phi_n^{-1}(t)$ as in Section \ref{subsectionholpunc}. Since $\C(u_n)$ is dense in $\C((u_n))=\C((t^{1/n}))$, by Remark \ref{remdensesubf} again, we may assume that $M_{\LL}\in \PGL_2(\C(u_n)).$ There is an open neighborhood $V$ of $o$ such that 
 $M_{\LL}$ and $M_{\LL}^{-1}$ are holomorphic on $V\setminus \{o\}$ i.e. they define holomorphic families $M_{V\setminus \{o\}}$ and $M^{-1}_{V\setminus \{o\}}.$
Then we conclude the proof by Proposition \ref{kiwirescaling}.
\endproof

\subsection{Endomorphisms without repelling type I periodic points}
In general the Julia set of an endomorphism $f_{\LL}$ on $\P_{\LL}^{1,\an}$ is a complicated object. The following theorem due to Favre-Rivera Letelier \cite{favrerivera}  and independently by Luo \cite[Proposition 11.4]{luo2019trees}  classifies the case when  $f_{\LL}$ has no repelling type I periodic points.
\begin{thm}\label{notype1}
Let $f_{\LL}:\P_{\LL}^{1,\an}\to \P_{\LL}^{1,\an}$ be an endomorphism. Assume $f_{\LL}$ has no type 1 repelling  periodic points. Then the Julia set of $f_{\LL}$ is contained in  a segment.
\end{thm}

By (v) of Proposition \ref{probasicfj}, $\sJ(f_{\LL})\neq \emptyset.$
In the above theorem, if $f_{\LL}$ does not have good reduction, then the segment can not be a point. 
As a corollary, we get the following lemma.
\begin{lem}\label{segmenttangent}
Let $f_{\LL}:\P_{\LL}^{1,\an}\to \P_{\LL}^{1,\an}$ be an endomorphism of degree $d\geq 2$, which does not have good reduction.  Assume that $\sJ(f_{\LL})$ is contained in  a minimal segment $[a,b]$.  Let $x$ be a repelling type 2 periodic point in  $(a,b)$ with period $q\geq 1$.  Then the tangent map $T_x f^q$ is conjugate to $z\mapsto z^{m}$ for some $|m|=\deg T_x f^q\geq 2$.  Moreover every bad direction of $T_x f^q$ is presented by $(x,a]$ or $(x,b]$ and under the above conjugacy, it is identified  to $0$ or $\infty$.
\end{lem}
\begin{proof}
Since $[a,b]$  is the minimal segment that contains $\sJ(f)$, $a$ and $b$ are contained in the Julia set.   Since $\deg f_{\LL}\geq 2$ and $f_{\LL}$ does not have good reduction, the Julia set is not a single point. Hence $a\neq b$. Let $v_1$ (resp. $v_2$) be the direction in $T_x$ represented by the segment $(x,a]$ (resp. $(x,b]$). 
Since $\sJ(f_{\LL})\subseteq [a,b]$, $\{v\in T_x|\,\, U(v)\cap \sJ(f_{\LL})\neq \emptyset\}=\{v_1,v_2\}.$
Since $\sJ(f_{\LL})$ is totally invariant, for $v\in T_x$, if $f^q(U(v))\cap \sJ(f_{\LL})\neq\emptyset$, then $U(v)\cap  \sJ(f_{\LL})\neq\emptyset.$ Hence $v\in \{v_1,v_2\}.$
This implies (i) $\left\{ v_1,v_2\right\}$ is totally invariant by $T_x f^q$.  Actually let $w\in (T_x f^q)^{-1}(v_i)$ for some $i=1,2$. Then we have $U(v_i)\subset f^q(U(w))$. This implies $f^q(U(w))\cap \sJ(f_{\LL})\neq \emptyset.$ Thus $w=v_i$.  (ii) Bad directions of $T_x f^q$ are contained in $\left\{ v_1,v_2\right\}$. Actually if $w$ is a bad direction, then we have $f^q(U(w))=\P_{\LL}^{1,\an}$, hence $f^q(U(w))\cap \sJ(f_{\LL})\neq\emptyset$, which implies $w=v_1$ or $v_2$.  

Finally,  an endomorphism of degree $\deg T_x f^q$ on $\P^1(\C)$ has a totally invariant set with two elements must is conjugate to $z\mapsto z^{m}$ for some $|m|=\deg T_x f^q$.  This conjugacy maps $ \left\{ v_1,v_2\right\}$ to $\left\{ 0,\infty\right\}$, which concludes the proof.
\end{proof}

The following Theorem is the main result of this section.
\begin{thm}\label{rescaling}
	Let $f_{\D}$ be a meromorphic family of degree $d\geq 2$ which is holomorphic on $\D^*$. Assume that  $f_{\D}$ does not have potentially good reduction at $0.$ For every $n\geq 1$, assume that the multipliers of the $n$-periodic points of $f_t$ are uniformly bounded in $t$.  
Then there is $n\geq 1, m\geq 2$, such that $g: z\mapsto z^m$ is a rescaling limit of $f_{U_n}|_V$ where $f_{U_n}$ is the base change of $f_{\D}$ by the morphism $U_n:=\D\to \D$ sending $t$ to $t^n$ as in Section \ref{subsectionholpunc} and $V$ is an open neighborhood of $o\in U_{n}$ isomorphic to $\D.$	
Moreover, we may ask the finite set $S$ in Definition \ref{defiresclim} to be contained in $\left\{ 0,\infty\right\}$.
\end{thm}
\begin{proof}
Let $f_{\LL}:\P_{\LL}^{1,\an}\to \P_{\LL}^{1,\an}$ be the endomorphism induced by $f_{\D}$. The multipliers of the $n$-periodic points of $f_t$ are uniformly bounded in $t$ implies $f_{\LL}$ has no repelling type 1 periodic points. By Theorem \ref{notype1}, $\sJ(f_{\LL})$ is contained in a minimal segment $[a,b]$.  
Since $f_{\D}$ does not have potentially good reduction at $0$, by Lemma \ref{lempotgoodr}, $f_{\LL}$ does not have good reduction. By a result of Rivera-Letelier \cite[Theorem 10.88]{baker2010potential}, there are infinitely many repelling type 2 periodic points. By (iii) of Proposition \ref{probasicfj}, they are necessarily contained in $\sJ(f_{\LL})$. Pick a repelling type II periodic point $x$ that are contained in $(a,b)$ of period $q\geq 1$. By Lemma \ref{segmenttangent}, replace $q$ by $2q$ if necessary, the tangent map $T_x f^q$ is conjugate to $z\mapsto z^{m}$ for some $m\geq 2$.  Moreover the bad directions of $T_x f^q$  can be identified  with a subset of $\left\{ 0,\infty\right\}$ by the conjugacy. The proof is finished by using Corollary \ref{cortypetwomodel}. 
\end{proof}

\section{A new proof of McMullen's theorem}\label{sectionmcmullen}
We can now give a new proof of Theorem \ref{thmmcmullen}.
\proof[Proof of Theorem \ref{thmmcmullen}]
Let $f_{\La}$ be a non-isotrivial stable irreducible algebraic family of endomorphisms of degree $d\geq 2$. 
Since $\La$ is covered by affine open subsets, we may assume that $\La$ itself is affine.
Cutting $\La$ by hyperplanes and removing the singular points,
we can reduce to the case that $\La$ is a connected Riemann surface of finite type.  
Since the only non-isotrivial family of exceptional endomorphisms of degree $d$ is the flexible Latt\`es family,
we only need to show that there is a nonempty open subset $W$ of  $\La$ such that for $t\in W$, $f_t$ is exceptional.

Write $\La=M\setminus B$, where $M$ is a compact Riemann surface and $B$ is a finite subset. Since $f_{\La}$ is algebraic, it extends to a meromorphic family of degree $d.$
We have $B(f_M)\subseteq B.$ Since $f_{\La}$ is not isotrivial, by Lemma \ref{lemeverygoodisotrivial}, there is $b\in B$ such that $f_M$ does not have potentially good reduction at $b.$ Reparametrize our family near $b\in M$, we get a meromorphic family $f_{\D}$ of degree $d\geq 2$, which is holomorphic on $\D^*$ and preserves multiplier spectrum.  

By Theorem \ref{rescaling}, after replacing $f_{\D}$ by the family $f_{V}$ in Theorem \ref{rescaling}, we may assume that $z\mapsto z^m$ for some $m\geq 2$ is a
rescaling limit of $f_{\D}$ with $S=\{0,\infty\}.$
Using the reformulation of the rescaling limit in Remark \ref{remresgeo}, there is an integer $q\geq 1$ and a meromorphic family $M_{\D}$ of degree 1, such that $M_{\D}$ and $M^{-1}_{\D}$ are holomorphic on $\D^*$, and $h_0$ is $z\to z^m$ where $h_\D:=M_\D^{-1}\circ f^q_{\D}\circ M_\D$ on $\P^1(\C)\times \D$.
Moreover $I(h_{\D})\subseteq \{0,\infty\}\times \{0\}\subseteq \P^1(\C)\times \D.$
We may replace $f_{\D}$ by $h_{\D}$ and assume that $f_0: z\mapsto z^m$ and $I(f_{\D})\subseteq \{0,\infty\}\times \{0\}\subseteq \P^1(\C)\times \D.$

The Julia set of $f_0$ is the unit circle $S^1$, and $f_0$  is  expanding on $S^1$. We need the following classical lemma of holomorphic motions of expanding sets. A proof can be found (without using quasiconformal maps) in Jonsson \cite{jonsson1998holomorphic}, which is also valid in higher dimension.  Let $K\subset \P^1(\C)$ be a compact set. We say $f:K\to K$ is \emph{expanding} if there exist $C>0$ and $\rho>1$ such that $|df^n(x)|\geq C\rho^n$ for every $n\geq 0$ and $x\in K$.
\begin{lem}\label{motion}
Let $(f_t)_{t\in\D}$ be a family of endomorphisms on $\P^1(\C)$. Suppose $f_0$ has an expanding set $K$, $f(K)=K$. Assume $(f_t)$ is a holomorphic family in a neighborhood of $K$, i.e. there exist an open set $V$, $K\subset V$ such that for every $z\in V$, $t\mapsto f_t(z)$ is holomorphic in $\D$. Then there exist $r>0$ and a continuous  map $h:\D_r\times K\to \P^1(\C)$ such that for each $t\in \D_r$:
\begin{points}
\item $K_t:=h(t,K)$ is an expanding set of $f_t$.
\item the map $h_t:=h(t, \cdot):K\to K_t$ is a homeomorphism and $f_t\circ h_t=h_t\circ f_0$.
\end{points}
\end{lem}

We set $f_0:z\mapsto z^m$ and $K:=S^1$ in the above lemma. The endomorphism $f_0$ has the following properties:
\begin{points}
\item[(1)] $f_0^{-1}(K)=f_0(K)=K;$
\item[(2)] all periodic points outside the exceptional set $\{0,\infty\}$ are contained in $K;$ 
\item[(3)] for every $n$-periodic point $z\in K$, we have $df_0^n(z)=m^n$. 
\end{points}
\medskip

Since the family $(f_t)_{t\in\D^\ast}$ has the same multiplier spectrum,  the multiplier of the periodic point $h_t(z)$ of $f_t$ does not change in the family $t\in \D_r^\ast$. Hence for every $t\in \D_r$ we have $df_t^n(h_t(z))=m^n$.  We choose a homoclinic orbit $o_i$, $i\geq 0$ of $f_0$ with $o_0=1$.  By $(1)$, all $o_i, i\geq 1$ are contained in $K.$
Hence $h_t(o_i)$, $i\geq 0$ is a homoclinic orbit of $f_t$ at $z=h_t(1)$, for $t\in \D_r$.  Let $q_i$, $i\geq 0$ be an adjoint sequence of $o_i$, $i\geq 0$. For every $t\in \D_r^\ast$, we need to show $h_t(q_i)$, $i\geq 0$ is an adjoint sequence of $h_t(o_i)$, $i\geq 0$.   In fact let $U_t$ be a linearization domain of $f_t$ at $h_t(1)$. Let $U_{t,i}$ be the connected component of $f_t^{-i}(U_t)$ containing $h_t(o_i)$. Let $l$ be a good return time of $U_t$. For every $n\geq l$, $f_t^n:U_{t,n}\to U_t$ is an isomorphism, with a unique fixed point $p_n$.  Let  $V$ be the connected component of $h_t^{-1}(U_t\cap K_t)$ containing $1$. It is an open arc in $S^1$. Let $V_n$ be the connected component of $f_0^{-n}(V)$ containing $o_n$. 
Since $K$ is totally invariant by $f_0$ and $V$ contains some linearization domain at $1$, 
after enlarging $l$ if necessary, for every $n \geq l$ we have $q_n\in V_n\cap K$. Hence $h_t(q_n)\in U_{t,n}\cap K_t$, which is  fixed by $f_t^n:U_{t,n}\to U_t$. By the uniqueness of $p_n$ we have $p_n=h_t(q_n)$. Hence $h_t(q_i)$, $i\geq 0$ is an adjoint sequence of $h_t(o_i)$, $i\geq 0$. 


For every $t\in \D_r^\ast$,  we consider the dynamics of $f_t$. The fixed point $h_t(1)$ has multiplier $m$ and the adjoint sequence $h_t(q_i)$, $i\geq 0$  of the  homoclinic orbit $h_t(o_i)$, $i\geq 0$  has multiplier $m^i$ when $i$ large enough. By Theorem \ref{thmadjconmultexcep}, $f_t$ is exceptional, which concludes the proof.
\endproof

\section{Conformal expanding repellers and applications}\label{sectioncer}
\subsection{Definition, examples and rigidity of CER}\label{subsectioncer}
The following definition was introduced by Sullivan \cite{sullivan1986quasiconformal}.
\begin{defi}\label{defcer}
Let $f:\P^1\to\P^1$ be an endomorphism over $\C$. An compact set $K\subset \P^1(\C)$ is called a CER of $f$ if
\begin{points}
\item there exist $m\geq 1$ and a neighborhood $V$ of $K$ such that $f^m(K)=K$ and $K=\cap_{n\geq 0}f^{-mn}(V)$.
\item $f^m:K\to K$ is expanding, i.e. there are constants $C>0$ and $\la>1$ such that $|df^{nm}(x)|\geq C\la^n$ for every $x\in K$ and $n\geq 1$;
\item $f^m:K\to K$ is topologically exact, i.e. for every open set $U\subset K$ there exist $n\geq 0$ such that $f^{mn}(U)=K$.
\end{points} 
\end{defi}

\begin{rem}\label{openmap}
Condition (i)+(ii) is equivalent to $f^m$ is expanding on $K$ and $f^m:K\to K$ is an open map \cite[Lemma 6.1.2]{przytycki2010conformal}.
\end{rem}
The following is an important class of examples of CER.
\begin{exa}\label{horseshoe}
Assume  $V$, $U_i$, $1\leq i\leq k$ are connected open sets in $\P^1(\C)$, $k\geq 2$  such that $\overline{U_i}\subset V$, and there exist $m\geq 1$ such that $f^m: U_i\to V$ is an isomorphism. Then we call
\begin{equation*}
K:=\left\{ z\in \bigcup_{i=1}^k U_i  \Bigg|\,\, f^{mn}(z)\in \bigcup_{i=1}^k U_i\;\text{for every}\;n\geq 0\right\}
\end{equation*}
a \emph{horseshoe} of $f$. We check that  $K$ satisfies the three conditions in Definition \ref{defcer}. Let $V_0:= \cup_{i=1}^k U_i $.
\begin{points}
\item[{\bf Condition (i):}] It follows from the definition of $K$; 
\item[{\bf Condition (ii):}]  $f^m: V_0 \to V$ strictly expands the hyperbolic metric of $V$,  this implies $f^m:K\to K$ is expanding;
\item[{\bf Condition (iii):}]Again using  $f^m: V_0 \to V$ strictly expands the hyperbolic metric of $V$, the maximal diameter of the connected components  of $f^{-nm}(V_0)\cap V_0$ shrinks  to $0$ when $n\to \infty$. For each open set $W\subset K$,  there exist integer $n\geq 0$ and a connected component $B$ of $f^{-nm}(V_0)\cap V_0$ such that $B\cap K\subset W$. Since $f^{(n+1)m}(B\cap K)=K$, we have $f^{(n+1)m}(W)=K$. Hence $f^m:K\to K$ is topologically exact.
\end{points}

 Moreover $K$ is a Cantor set, in particular $K$ is not a finite set. 
\end{exa}

When $f$ has degree at least $2$, there are plenty of horseshoes. Following the terminology in section 2, 
we can construct a horseshoe associated to finite numbers of homoclinic orbits at $o$. We prove the following lemma which will be used in the proof of Theorem \ref{markedlength}. 
\begin{lem}\label{cerhomoclinic}
Let $o$ be a repelling fixed point. Let $k\geq 1$ be an integer. Assume for each fixed $1\leq j\leq k$,  $o_i^{j}$, $i\geq 0$ is a homoclinic orbit of $o$ such that $o_i^j\notin C(f)$. Then there exist an integer $m\geq 1$ and a horseshoe $f^m:K\to K$ such that $o_{mi}^j\in K$ for every $i\geq 0$ and $1\leq j\leq k$. Moreover for each $0\leq q\leq m-1$, $f^q(K)$ is a CER.
\end{lem}
\begin{proof}
By Lemma \ref{proexihors}, there exist a linearization domain $U$ of $o$ and an integer $m$ such that for every $1\leq j\leq k$, $m$ is a common good return time of $U$ for the homoclinic orbits $o_i^j$, $i\geq 0$. Let $U_m^j$ be the connected component of $f^{-m}(U)$ containing $o_m^j$.  Let 
\begin{equation*}
V_0:= \left(\bigcup_{j=1}^k U_m^j\right)\cup g^m(U).	
\end{equation*} Then the set
\begin{equation*}
K:=\left\{ z\in V_0 |\,\, f^{mn}(z)\in V_0\;\text{for}\;n\geq 0\right\}
\end{equation*}
is a horseshoe of $f$.  Clearly we have $o_{mi}^j\in K$ for every $i\geq 0$ and $1\leq j\leq k$. 
\medskip

For each $0\leq q\leq m-1$ let $K_q:=f^q(K)$.  We know that $f^q: U_m^j\to U_{m-q}^j$ is an isomorphism, and $f^q:g^m(U)\to g^{m-q}(U)$ is an isomorphism. This implies $f^q:V_0\to f^q(V_0)$ is a finite holomorphic covering (the image of $f^q$ of two components of $V_0$ may coincide).  We let $\phi_q$ denote this map.  Then we have 
\begin{equation*}
	\phi_q\circ f^m|_{V_0}=f^m|_{f^q(V_0)}\circ \phi_q 
\end{equation*}
on $f^{-m}(V_0)\cap V_0$, which implies that  $f^m:K\to K$ and $f^m:K_q\to K_q$ are holomorphically semi-conjugated by $\phi_q$ on the corresponding neighborhoods of $K$ and $K_q$.  We check that $K_q$  satisfies the three conditions in Definition \ref{defcer}.
Since $\phi_q$ is a covering and $f^m:K\to K$ is an open map, $f^m:K_q\to K_q$  is an open map.
Since $f^m:K\to K$ is expanding and $|d\phi_q|>c$ on $K$ for some constant $c>0$, $f^m:K_q\to K_q$ is expanding. 
By Remark \ref{openmap}, conditions (i) and (ii) hold. Since $f^m:K\to K$ is topologically exact  and $\phi_q:K\to K_q$ is a semi-conjugacy, $f^m:K_q\to K_q$ is topologically exact. This implies Condition (iii). Hence $K_q=f^q(K)$ is a CER.
\end{proof}



\medskip

The following definition of linear CER was introduced by Sullivan \cite{sullivan1986quasiconformal}.
\begin{defi}\label{deflinearcer}
Let $f:\P^1\to\P^1$ be an endomorphism over $\C$. Let $K$ be a CER of $f$. $f(K)=K$.  We call $K$ linear if one of the following conditions holds. 
\begin{points}
\item The function $\log |df|$ is cohomologous to a locally constant function on $K$, i.e. there exist a continuous function $u$ on $K$ such that $\log |df|-(u\circ f-u)$ is locally constant on $K$.
\item there exist an atlas $\left\{ \phi_i\right\}_{1\leq i\leq k}$ that is a family of holomorphic injections $\phi_i:V_i\to \C$ such that $K\subset \cup_{i=1}^k V_i$ and all the maps $\phi_i\circ\phi_j^{-1}$ and $\phi_i\circ f\circ \phi_j^{-1}$ are affine. 
\end{points}
\end{defi}
A proof that these two conditions are actually equivalent can be found in Przytycki-Urbanski \cite[section 10.1]{przytycki2010conformal}.
\medskip

The following Sullivan's rigidity theorem \cite{sullivan1986quasiconformal} will be used in the proof of Theorem \ref{thmlength} and Theorem \ref{markedlength}.  A proof can be found in \cite[section 10.2]{przytycki2010conformal}.

\begin{thm}[Sullivan]\label{cerrigidity}
Let $(f,K_f)$, $(g,K_g)$ be two CERs such that $K_f$ is non-linear, $f(K_f)=K_f$, $g(K_g)=K_g$.  Let $h:K_f\to K_g$ be a homeomorphism such that $h\circ f=g\circ h$ on $K_f$. Then the following two conditions are equivalent
\begin{points}
	
	\item for every periodic point $x\in K_f$ we have $|df^n(x)|=|dg^n(h(x))|$, where $n$ is the period of $x$;

	\item there exist a neighborhood $U$ of $K_f$ and a neighborhood $V$ of $K_g$ such that $h$ extends to a conformal map $h:U\to V$. 
\end{points}
\end{thm}

Here as in Theorem \ref{markedlength}, a conformal map may change the orientation of $\P^1(\C)$.
\medskip

\subsection{Having a linear CER implies exceptional}
Now we give a proof of Theorem \ref{linearcer}. 
\proof[Proof of Theorem  \ref{linearcer}]
Let $K$ be a linear CER of $f$, which is not a finite set. By \cite[Proposition 4.3.6]{przytycki2010conformal}, there exist a repelling periodic point $o\in K$ of $f$.  Passing to an iterate of $f$ we may assume $f(K)=K$ and $f(o)=o$.  Topological exactness of $f$ on $K$ implies for every $a\in K$, the preimages of $f|_K$ are  dense in $K$.  Let $U$ be a linearization domain $U$ of $f$ at $o$.  Since $K\neq \left\{ o\right\}$,  there exist $l\geq 1$ and a point $p_l\in K$ such that  $p_l\neq o$, $f^l(p_l)=o$.
 Then there exist a (unique) homoclinic orbit $o_i$, $i\geq 0$ such that $o_l=p_l$ and $o_i\in U$ for every $i\geq l$.  Clearly $o_i\in K$ when $i\leq l$.  By the definition of CER, there exist a neighborhood $V$ of $K$ such that  $K=\cap_{n\geq 0}f^{-n}(V)$. Shrink $U$ if necessary we assume $U\subset V$.  Hence for every $i\geq l$  we have $o_i\in V$. This implies for every fixed $i\geq 0$, for every $n\geq 0$ we have $f^n(o_i)\in V$. Hence $o_i\in K$ for every $i\geq 0$. 
\medskip

Let $\left\{V_j\right\}_{1\leq j\leq k}$ be an affine atlas in Definition \ref{deflinearcer}.  Shrink the linearization domain  $U$ if necessary we may assume for every $i\geq 0$,  $U_i$ (the connected component of $f^{-i}(U)$ containing $o_i$)  is contained in some affine chart, say $V_{j(i)}$. In particular $U\subset V_{j(0)}$ and $U_i\subset  V_{j(0)}$ for every $i\geq l$.  Let $\left\{ q_i\right\}$, $i\geq 0$ be the adjoint sequence of $o_i$, $i\geq 0$.  For every large enough integer $n$ we have $q_n\in U_n$. For such fixed $n$, for every  $1\leq i\leq n$ we have $f^{n-i}(q_n)\in U_i\subset V_{j(i)}$.  Let $\la_i\in \C^\ast$  be the derivatives of the affine map $\phi_{j(i+1)}\circ f\circ \phi_{j(i)}^{-1}$, where $0\leq i\leq l-1$.  Let $\la\in \C^\ast$ be the derivatives of the affine map  $\phi_{j(0)}\circ f\circ \phi_{j(0)}^{-1}$. Then we have $df(o)=\la$, and for every $n$ large enough we have
\begin{equation*}
df^n(q_n)=\left(\prod_{i=0}^{l-1} \la_i \right)\la^{n-l}.
\end{equation*}
By Theorem \ref{thmadjconmultexcep}, $f$ is exceptional. The proof is finished.
\endproof

\subsection{Marked length spectrum rigidity}
We now prove Theorem \ref{markedlength} by using Theorem \ref{linearcer} and Lemma \ref{cerhomoclinic}. 
\proof[Proof of Theorem \ref{markedlength}]
It is clear that (ii) implies (i). We need to show (i) implies (ii). Assume that $h$ preserves the marked length spectrum on $\Omega$. 
If $h$ extends to a global conformal map $\P^1(\C)\to \P^1(\C)$, since $h\circ f=g\circ h$ on $\sJ(f)$, the same equality holds on $\P^1(\C)$.
So we may replace $f$ by its iterate.
Passing to an iterate of $f$, we assume $f$ has a repelling fixed point $o\in \Omega$ and $o\notin PC(f)$.  A result of Eremenko-van Strien \cite{eremenko2011rational} says that  if a  non-Latt\`es endomorphism $f$ has the property that all the multipliers are real for periodic points contained in a non-empty open set of $\sJ(f)$, then $\sJ(f)$ is contained in a circle.  By this result there are two cases:
\begin{points}
	\item we can further choose $o$ such that $df(o)\notin \R$;
	\item $\sJ(f)$ is contained in a circle $C$.
\end{points}

By our choice of $o$, $h(o)$ is a repelling fixed point of $g$. Moreover we have $h(o)\notin PC(g)$ since $h$ preserve critical points in the Julia set. This can be proved using the total invariance of the Julia sets and the fact that critical means locally not injective.  Let $o_i$ $i\geq 0$ be a homoclinic orbit of $o$. Then $h(o_i)$, $i\geq 0$ is a homoclinic orbit of $h(o)$. Let $U$ be a linearization domain of $o$ such that $U\cap \sJ(f)\subset \Omega$. Let $W$ be a connected open neighborhood of $h(o)$ such that $h(U\cap \sJ(f))\subset W$ and $W\cap \sJ(g)\subset h(\Omega)$.  By Lemma \ref{proexihors}, shrink $U$ and $W$ if necessary there exist  $m\geq 1$ such that  $m$ is a good return time of $U$ (resp. $W$) for $o_i, i\geq 0$ (resp. $h(o_i), i\geq 0$ ). By Lemma \ref{cerhomoclinic}, there exist two horseshoes,  $f^m:K_f\to K_f$ (resp. $g^m:X_g\to X_g$) such that $o_{im}\in K_f$, $i\geq 0$ (resp. $h(o_{im})\in X_g$, $i\geq 0$).  We let $K_g:=h(K_f)$. By our construction we have $h:K_f\to K_g$ is a homeomorphism and $h\circ f^m=g^m\circ h$ on $K_f$. Moreover $K_g\subset X_g$. We check that $K_g$ is a CER of $g$: $g^m:K_g\to K_g$ is open and topologically exact since $f^m:K_f\to K_f$ is; $g^m:K_g\to K_g$ is expanding since $K_g$ is contained in an expanding set $X_g$. Hence $K_g$ is a CER of $g$.  Passing to an iterate we may assume $f(K_f)=K_f$ and $g(K_g)=K_g$.  To simplify the notation, for $i\geq 0$ we let $o_i$ be the unique point in $f^{-i}(o)$ which is contained in the previous homoclinic orbit.
\medskip

Since $f$ is not exceptional, $K_f$ is a non-linear CER by Theorem \ref{linearcer}.  Moreover by our construction we have $K_f\subset \Omega$.   Hence for every $n$-periodic point $x\in K_f$, we have $|df^n(x)|=|dg^n(h(x))|$.  By Theorem \ref{cerrigidity}, $h$ can be extended conformally to a neighborhood $V$ of $K_f$ such that $V\cap \sJ(f) \subset \Omega$. We denote this extension by $\tilde{h}$.  In case (ii), we can further assume that $\tilde{h}$ is in fact holomorphic: if  $\tilde{h}$ is antiholomorphic on some connected component $B$ of $V$, let $\phi$ be a  non-identity conformal map (necessarily antiholomorphic) on $\P^1(\C)$ such that $\phi$ fixes every point in $C$, then on $B$, we may replace $\tilde{h}$ by $\tilde{h}\circ \phi$, which is holomorphic. We have $\tilde{h}=h$ on $K_f$.  Since $\tilde{h}\circ f=g\circ \tilde{h}$ on $K_f$ and $K_f$ is a perfect set, by the conformality of $\tilde{h}$ we have $\tilde{h}\circ f=g\circ \tilde{h}$ on $V$. 
\medskip

Next we show that $\tilde{h}=h$ on $U_0\cap \sJ(f)$, where $U_0\subset V$ is a linearization domain of $o$. Let $E$ be the set of all $f$-preimages of $o$.  For every $a\in E\cap U_0$, $f^q(a)=o$, there exist a homoclinic orbit $o_i'$ of $o$ such that $a=o'_q$, and $o'_i\in U_0$ for every $i\geq q$.

Choose $m'\geq q$, by similar construction as in first paragraph, we get two CERs,  $f^{m'}:K''_f\to K''_f$ (resp. $g^{m'}:K''_g\to K''_g$) such that $o_{im'}\in K''_f$ and $o_{im'}'\in K''_f$ (resp. $h(o_{im'})\in K''_g$ and $h(o'_{im'})\in K''_g$) for $i\geq 0$. Moreover, $K''_f$ is a horseshoe and $K''_g$ is contained in a horseshoe $X''_g$.
By Lemma \ref{cerhomoclinic}, $K'_f:=f^{m'-q}(K''_f)$ and $f^{m'-q}(X''_f)$ are CERs.
Since $K'_g:=f^{m'-q}(K''_g)\subseteq f^{m'-q}(X''_f)$, $g^m: K'_g\to K'_g$ is expending. 
Since $h:K'_f\to K'_g$ is a homeomorphism and $h\circ f^{m'}=g^{m'}\circ h$ on $K'_f$,
$g^{m'}: K'_g\to K'_g$ is open and topologically exact. By Remark \ref{openmap}, $K'_g$ is a CER.
Moreover, we have $o_{q+im'}\in K'_f$ and $o_{q+im'}'\in K'_f$ (resp. $h(o_{q+im'})\in K'_g$ and $h(o'_{q+im'})\in K'_g$) for $i\geq 0$. 
Since $f$ is not exceptional, $K'_f$ is a non-linear CER by Theorem \ref{linearcer}.  Moreover every periodic point $x$ of $f^{m'}:K'_f\to K'_f$ has the form $x=f^{m'-q}(y)$, where $y$ is a  periodic point $x$ of $f^{m'}:K''_f\to K''_f$. Since $K''_f\subset \Omega$, we get that the $f$-orbit of $x$ has non empty intersection with $\Omega$. This implies for every  $n$-periodic point $x$ of $f^{m'}:K'_f\to K'_f$ we have $|df^{m'n}(x)|=|dg^{m'n}(h(x))|$. By Theorem \ref{cerrigidity}, $h$ can be extended conformally to a neighborhood $V'$ of $K'_f$.  Denote this extension by $\tilde{h}'$. In case (ii) we further assume that $\tilde{h}'$ is  holomorphic. We have  $\tilde{h}'(o_{q+im'})=\tilde{h}(o_{q+im'})=h(o_{q+im'})$, $i\geq 0$.  The set  $\left\{o_{q+im'}, i\geq 0 \right\}$ is a set with accumulation point $o$. We claim that $\tilde{h}'=\tilde{h}$ on $V_0$,  where $V_0$ is the connected component of $V\cap V'$ containing $o$.  In case (i), since $df(o)\notin \R$, $\tilde{h}'$ and $\tilde{h}$ are both holomorphic or both antiholomorphic on $V_0$, hence $\tilde{h}'=\tilde{h}$ on $V_0$. In case (ii), by our choices $\tilde{h}'$ and $\tilde{h}$ are both holomorphic, hence $\tilde{h}'=\tilde{h}$ on $V_0$. 

there exist $b\in V_0\cap K'_f$  such that $f^{q+nm'}(b)=a$ for some $n\geq 0$ and $\left\{ b, f(b), \cdots, f^{q+nm'}(b)\right\}\subset U_0$.  We also have $\tilde{h}(b)=\tilde{h}'(b)=h(b)$. Since $\tilde{h}\circ f=g\circ \tilde{h}$ on $U_0$ we have
\begin{equation*}
	\tilde{h}(a)=\tilde{h}(f^q(b))=g^q(\tilde{h}(b))=g^q(h(b))=h(f^q(b))=h(a).
\end{equation*}
This implies $\tilde{h}=h$ on $E\cap U_0$. Since $E$ is dense in $\sJ(f)$ we get that $\tilde{h}=h$  on $U_0\cap \sJ(f)$. 
\medskip

To summary we have shown that the homeomorphism $h:\sJ(f)\to \sJ(g)$ conjugates $f$ to $g$ can be  extended conformally to a disk intersecting $\sJ(f)$. By a lemma due to Przytycki-Urbanski \cite[Proposition 5.4, Lemma 5.5]{przytycki1999rigidity}, $h$ extends to a conformal map $h:\P^1(\C)\to \P^1(\C)$ such that  $h\circ f=g\circ h$  on $\P^1(\C)$. 
\endproof


\subsection{Marked multiplier spectrum rigidity}
Combining Theorem \ref{markedlength} and Eremenko-van Strien's theorem \cite{eremenko2011rational}, we now prove Theorem \ref{markedmultiplier}.
\proof[Proof of Theorem \ref{markedmultiplier}]
It is clear that (ii) implies (i). We need to show that (i) implies (ii). 
Assume $h$ preserves the marked multiplier spectrum on $\Omega$. By Theorem \ref{markedlength}, $h$ can be extended to a conformal map on $\P^1(\C)$. If $h$ is holomorphic then we are done. If $h$ is antiholomorphic, then the multipliers of all periodic points in $\Omega$ are real. By the main theorem in \cite{eremenko2011rational}, $\sJ(f)$  is contained in a circle $C$.  Let $\phi$ be a  non-identity conformal map on $\P^1(\C)$ such that $\phi$ fixes every point in $C$.  Let $\tilde{h}:=h\circ \phi$, then $\tilde{h}\in \PGL_2(\C)$, and we have $ \tilde{h}\circ f=g\circ \tilde{h}$ on $\P^1(\C)$, this finishes the proof.
\endproof


\subsection{Another proof of McMullen's theorem}
Now we can give another proof of Theorem \ref{thmmcmullen} using $\la$-Lemma and Theorem \ref{markedmultiplier}. 

\proof[Proof of Theorem \ref{thmmcmullen}]
By using $\la$-Lemma \cite[Theorem 4.1]{mcmullen2016complex}, it is well known that two endomorphisms in  a
stable family are quasiconformally conjugate on thier Julia sets. Assume by contradiction the conclusion is not true. Since  exceptional endomorphisms that are not flexible Latt\`es are isolated in the moduli space $\mathcal{M}_d$,  there is at least one $f$  in the familly that is not exceptional.  Let $g$ be another endomorphism in the family. Let $h:\sJ(f)\to \sJ(g)$ be the quasicoformal conjugacy.  Since multiplier spectrum is preserved in this family
 and the conjugacy $h$ moves continuously in the family, for every $n$-periodic point $x$ of $f$ we have $df^n(x)=dg^n(h(x))$. By Theorem \ref{markedmultiplier}, $h$ extends to an automorphism on $\P^1(\C)$. This contradicts to the assumption that the family  is non-isotrivial. 
\endproof

\medskip
\subsection{Milnor's conjecture on Lyapunov exponent}
We now prove Theorem \ref{thmmilnor1}  using Theorem \ref{linearcer}.
\proof[Proof of Theorem \ref{thmmilnor1}]
Let $S$ be the finite exceptional set of periodic points in Theorem \ref{thmmilnor1}.  Passing to an iterate of $f$ there exist a repelling fixed point $o$ of $f$ such that $o\notin S$.  Choose a linearization domain $U$ of $o$ such that $U\cap S=\emptyset$. By the discussion in Lemma \ref{cerhomoclinic}, there exist a horseshoe $K\subset U$.  Passing to an iterate of $f$,  we assume that $f(K)=K$. For every $n$-periodic point $x\in K$, we have $|df^n(x)|=b^n$ for some $b>0$.  Consider the function $\phi:=\log |df|$.  We have shown that for every $n$-periodic point $x\in K$,  $\sum_{i=0}^{n-1}\phi(f^i(x)) =n\log b$. Recall the following classical Livsic Theorem \cite{livvsic1972cohomology}
\begin{lem}
	Let $K$ be a CER of $f$, $f(K)=K$. Let $\phi$ be a H\"older continuous function on $K$. Assume there exits a constant $C$ such that for every $n$-periodic point $x\in K$ of $f$ we have 
	\begin{equation*}
		\sum_{i=0}^{n-1}\phi(f^i(x))=nC,
	\end{equation*}
	then  there exist a continuous function $u$ on $K$ such that $\phi-C=u\circ f-u$. 
\end{lem} 
Applying the above theorem to $\phi:=\log |df|$, we get that $\phi$ is cohomologous to a constant function on $K$ in the sense of Definition \ref{deflinearcer}. Hence $K$ is a linear CER, which is not a finite set. By Theorem \ref{linearcer}, $f$ is exceptional. The proof is finished.
\endproof

Next we prove Corollary \ref{lyapunovspectrum}. Let $f:\P^1\to \P^1$ be an endomorphism over $\C$ of degree at least $2$.  By Gelfert-Przytycki-Rams \cite{gelfert2010lyapunov}, there is a forward invariant finite set $\Sigma\subset \sJ(f)$ with cardinality at most 4 (possibly empty), such that for every finite set $F\subset \sJ(f)\setminus \Sigma$, we have $f^{-1}(F)\setminus C(f)\neq F$.  Let $\Delta'(f)$ be  the closure of the Lyapunov exponents of periodic points contained in $\sJ(f)\setminus \Sigma$. The following Theorem was proved by Gelfert-Przytycki-Rams-Rivera Letelier. Be aware that the definition of ``exceptional" in \cite{gelfert2010lyapunov} and \cite{gelfert2013lyapunov} has a different meaning.
\begin{thm}\cite[Theorem 2]{gelfert2010lyapunov}, \cite[Theorem 1, Proposition 10]{gelfert2013lyapunov}.\label{gprr}
Let $f:\P^1\to \P^1$ be an endomorphism over $\C$ of degree at least $2$. Then $\Delta'(f)$ is a  closed interval  (possibly a singleton).
\end{thm} 

\proof[Proof of Corollary \ref{lyapunovspectrum}]
If $\Delta'(f)$ is not a singleton we are done by Theorem \ref{gprr}.  If $\Delta'(f)$ is a singleton, then by Theorem \ref{thmmilnor1}, $f$ is exceptional, contradics to our assumption. This finishes the proof.
\endproof
\medskip

\subsection{A simple proof of Zdunik's theorem}
Next we give a simple proof of Theorem \ref{zdunik}, using Theorem \ref{linearcer}. 
\proof[Proof of Theorem \ref{zdunik}]
It is easy to observe that if $f$ is exceptional then $\mu$ is absolutely continuous with respect to $\Lambda_\alpha$. We only need to show the converse is true. 
\medskip

Let $\phi:=\alpha \log |df|$. Following Zdunik \cite{zdunik1990parabolic}, we say $\phi$ is cohomologous to $\log d$ if there exist a function $u\in L^2 (\sJ(f),\mu)$ such that $\phi-\log d=u\circ f-u$ holds for almost every point, where $\sJ(f)$ is the Julia set. By a result of Przytycki-Urbanski-Zdunik \cite[Theorem 6]{przytycki1989harmonic}, $\phi$ is not cohomologous to $\log d$ implies $\mu$ is singular with respect to $\Lambda_\alpha$.  So we only need to show that $\phi$ is cohomologous to $\log d$ implies $f$ is exceptional.  Now we assume $\phi-\log d=u\circ f-u$ for some $u\in L^2 (\sJ(f),\mu)$. By a lemma due to Zdunik \cite[Lemma 2]{zdunik1990parabolic}, for every $p\notin PC(f)$, there exist a neighborhood $U$ of $p$ such that $u$ equals to a continuous function almost everywhere. We observe that if $\phi_f:=\alpha \log |df|$ satisfy $\phi_f-\log d=u\circ f-u$  then $\phi_{f^n}:=\alpha \log |df^n|$ satisfies 
\begin{equation}\label{7.1}
	\phi_{f^n}-n\log d=u\circ f^n-u.
\end{equation}
\medskip 

Passing to an iterate of $f$ there exist a repelling fixed point $o\notin PC(f)$. Let $U$ be a linearization domain of $o$ such that $u$ is continous on $U$.  Let $K$ be a horseshoe of $f$ contained in $U$. Passing to an iterate of $f$, we may assume $f(K)=K$. Since $u$ is continuous on $K$, by (\ref{7.1}) the function $\log |df|$ is cohomologous to a constant on $K$ in the sense of  Definition \ref{deflinearcer}. This implies $K$ is a linear CER.  Since $K$ is not a finite set, by Theorem \ref{linearcer}, $f$ is exceptional. The proof is finished. 
\endproof

\section{Length spectrum as moduli}\label{sectionlength}
For $N\geq 1$, the symmetric group $S_N$ acts on $\C^N$ (resp. $\R^N$) by permuting the coordinates. 
Using symmetric polynomials, one can show that $\C^N/S_N\simeq \C^N.$
For every element $(\la_1,\dots, \la_N)\in \C^N$ (resp. $\R^N$), we denote by $\{\la_1,\dots,\la_N\}$ its image in $\C^N/S_N$ (resp. $\R^N/S_N$).
We may view the elements in $\C^N/S_N$ as multisets.\footnote{A multiset is a set except allowing multiple instances for each of its elements. The number of the instances of an element is called the multiplicity. For example: $\{a,a,b,c,c,c\}$ is a multiset of cardinality $6$, the multiplicities for $a,b,c$ are 2,1,3.}

\medskip

For $d\geq 2,$
let $f_{\text{Rat}_d}: \text{Rat}_d\times \P^1$ be the endomorphism sending $(t,z)$ to $(t, f_t(z))$ where $f_t$ is the endomorphism  associated to 
$t\in \text{Rat}_d$.  For  $t\in \text{Rat}_d$, $f_t^n$ has $N_n:=d^n+1$ fixed points counted with multiplicities. 
Let $\la_1,\dots,\la_{d^n+1}$  be the multipliers of such fixed points.
Define $s_n(t)=s_n(f_t):=\{\la_1,\dots,\la_{d^n+1}\}\in \A^{N_n}/S_{N_n}$ the \emph{$n$-th multiplier spectrum of $f_t$}.
Similarly, define $L_n(t)=L_n(f_t):=\{|\la_1|,\dots,|\la_{d^n+1}|\}\in \R^{N_n}/S_{N_n}$ the \emph{$n$-th length spectrum of $f_t$}.
Both $s_n(f_t)$ and $L_n(f_t)$ only depend on the conjugacy class of $f_t.$

\medskip

For every $n\geq 1,$ let $\Per_n(f_{\text{Rat}_d})$ be the closed subvariety of $\text{Rat}_d\times \P^1$ of the $n$-periodic points of $f_{\text{Rat}_d}.$ 
Let $\phi_n: \Per_n(f_{\text{Rat}_d})\to \text{Rat}_d$ be the first projection. It is a finite map of degree $d^n+1.$
Let $\la_n: \Per_n(f_{\text{Rat}_d})\to \A^1$ be the algebraic morphism $(f_t, x)\mapsto df_t^n(x)\in \A^1$. 
Let $|\la_n|: \Per_n(f_{\text{Rat}_d(\C)})(\C)\to [0,+\infty)$ be the composition of $\la_n$ to the norm map $z\in \C\mapsto |z|\in [0,+\infty).$ 
A fixed point $x$ of $f_t^n$ has multiplicity $>1$ if and only if $df_t^n(x)=1$. This shows that the map $\phi_n$ is \'etale at every point $x\in \Per_n(f_{\text{Rat}_d})\setminus \la_n^{-1}(1).$

\medskip

We may view $\Per_n(f_{\text{Rat}_d})$ as the moduli space of endomorphisms of degree $d$ with a marked $n$-periodic point.
So we may also denote  it by $\text{Rat}_d[n]$ or $\text{Rat}^1_d[n]$.
More generally, for every $s=1,\dots, d^n+1$, one may construct the moduli space $\text{Rat}^s_d[n]$ of endomorphisms of degree $d$ with $s$ marked $n$-periodic point as follows:
For $s=2,\dots, d^n+1,$
consider the fiber product $(\text{Rat}_d[n])^s_{/\text{Rat}_d}$ of $s$ copies of $\text{Rat}_d[n]$ over 
$\text{Rat}_d.$ For $i\neq j\in \{1,\dots, d^n+1\}$, let 
$\pi_{i,j}: (\text{Rat}_d[n])^s_{/\text{Rat}_d}\to (\text{Rat}_d[n])^2_{/\text{Rat}_d}$ be the projection to the $i,j$ coordinates. 
The diagonal $\Delta\subseteq (\text{Rat}_d[n])^2_{/\text{Rat}_d}$ is an irreducible component of $(\text{Rat}_d[n])^2_{/\text{Rat}_d}$.
One define $\text{Rat}^s_d[n]$ to be the Zariski closure of $$(\text{Rat}_d[n])^s_{/\text{Rat}_d}\setminus (\cup_{i\neq j\in \{1,\dots, d^n+1\}}\pi_{i,j}^{-1}(\Delta))$$ in $(\text{Rat}_d[n])^s_{/\text{Rat}_d}.$ Denote by $\phi_n^s: \text{Rat}^s_d[n]\to \text{Rat}_d$ the morphism induced by $\phi_n$. Let $\la^s_n: \text{Rat}^s_d[n]\to \A^s$ the morphism
defined by $(t, x_1,\dots,x_s)\mapsto (df^n(x_1),\dots, df^n(x_s))$ and  $|\la^s_n|: \text{Rat}^s_d[n](\C)\to \R^s$ the map
defined by $(t, x_1,\dots,x_s)\mapsto (|df^n(x_1)|,\dots, |df^n(x_s)|).$ 
Since $\phi_n$ is \'etale at every point $x\in \Per_n(f_{\text{Rat}_d})\setminus \la_n^{-1}(1),$
$\phi^s_n$ is \'etale at every point $x\in (\la_n^s)^{-1}((\A^1\setminus \{1\})^s).$ 

\medskip

To prove Theorem \ref{thmlength}, we need to study the subsets taking form 
$\La_n(a):=L_n^{-1}(a)$ where $a\in \R^{N_n}/S_{N_n}.$
Since $L_n$ is not holomorphic (hence not algebraic), in general, the above set is not algebraic. The problem is that one projects a real algebraic set under a finite map may not be real algebraic.
To get some algebricity of $\La_n(a)$, one can view $\text{Rat}_d(\C)$ as an real algebraic variety by 
splitting a complex variable $z$ into two real variety $x,y$ via $z=x+iy$. A more theoretic way to do this is 
using the notion of Weil restriction. See Section \ref{sectionweilrest} for a brief introduction.
However, even when we view $\text{Rat}_d(\C)$ as a real algebraic variety, $\La_n(a)$ is not real algebraic in general (c.f. Theorem \ref{thmexamplenonalg}).
Here \emph{real algebraic} means Zariski closed when viewing $\text{Rat}_d(\C)$ as a real algebraic variety. See Section \ref{sectionweilrest} for the precise definition. 
This is one of the main difficulty in the proof of Theorem \ref{thmlength}.
To solve this problem, we introduce a class of closed subsets of $\text{Rat}_d(\C)$ that are images of algebraic subsets under \'etale morphisms. We will study such subsets in 
Section \ref{subsectionimageetale}.

\subsection{An example of a length level set which is not real algebraic}
The main result of this section is Theorem \ref{thmexamplenonalg}, in which we give an example to show that the subsets $\La_n(a)$ may not be real algebraic in $\text{Rat}_d(\C)$\footnote{In our example, we will take $d=2$ and $n=1$.}.

Except Definition \ref{defirealzt}, in which we give a precise definition of the notion \emph{real algebraic} using Weil restriction, this section will not be used in the rest of the paper.


\subsubsection{Weil restriction}\label{sectionweilrest}
We briefly recall the notion of Weil restriction.  See \cite[Section 4.6]{Poonen2017} and \cite[Section 7.6]{Bosch1990} for more information. 

\medskip

Denote by $Var_{/\C}$ (resp. $Var_{/\R}$) the category of varieties over $\C$ (resp. $\R$).
For every variety $X$ over $\C$, there is a unique variety $R(X)$ over $\R$ represents the functor $Var_{/\R}\to Sets$
sending $V\in Var_{/\R}$ to $\Hom(V\otimes_{\R}\C, X).$
It is called the \emph{Weil restriction of $X$}. The functor $X\mapsto R(X)$ is called the Weil restriction.
One has the canonical morphism 
$\tau_X: X(\C)\to R(X)(\R),$  which is a real analytic diffeomorphism. One may view $X(\C)$ as a real algebraic variety via $\tau_X.$
\begin{defi}\label{defirealzt}The \emph{real Zariski topology} on $X(\C)$ is the restriction of the Zariski topology on $R(X)$ via $\tau_X$.
A subset $Y$ of $X(\C)$ is \emph{real algebraic} if it is closed in the real Zariski topology.
\end{defi}
By (iii) of Proposition \ref{probasicweil} below, the  real Zariski topology is stronger than the Zariski topology on $X(\C).$

\medskip

Roughly speaking, the Weil restriction is just constructed  by
splitting a complex variable $z$ into two real variables $x,y$ via $z=x+iy$.
For the convenience of the reader, in the following example, we show the concrete construction of $R(X)$  when $X$ is affine. 

\begin{exe}
First assume that $X=\A^N_{\C}$. Then $R(X)=\A^{2N}_{\R}.$ The map $$\tau_{X}: \A^N_{\C}(\C)=\C^N\to \A^{2N}_{\R}(\R)=\R^{2N}$$ sends $(z_1, \dots, z_N)$ to $(x_1, y_1, x_2,y_2,\dots, x_N,y_N)$ where $z_j=x_j+iy_j$.

\medskip
Consider the algebra $\B:=\C[I]/(I^2+1)\simeq \C\oplus I\C$.
Every $f\in \C[z_1,\dots,z_N]$ defines an element $$F:=f(x_1+Iy_1, \dots, x_N+Iy_N)\in \B[x_1,y_1,\dots, x_N,y_N].$$
Since $$\B[x_1,y_1,\dots, x_N,y_N]=\C[x_1,y_1,\dots, x_N,y_N]\oplus I\C[x_1,y_1,\dots, x_N,y_N],$$ $F$ can be uniquely decomposed to $F=r(f)+Ii(f)$ where $r(f), i(f)\in \C[x_1,y_1,\dots,x_N,y_N].$

If $X$ is the closed subvariety of $\A^N_\C=\Spec\C[z_1,\dots, z_M]$ defined by the ideal $(f_1,\dots, f_s)$, then $R(X)$ is the closed subvariety of 
$R(\A^N_{\C})=\A^{2N}_{\R}=\Spec \R[x_1,y_1,\dots, x_N,y_N]$ defined by the ideal generated by $r(f_1), i(f_1),\dots, r(f_s), i(f_s)$.
\end{exe}

\medskip

We list some basic properties of Weil restriction without proof.
\begin{pro}\label{probasicweil}Let $X, Y\in Var_{/\C}$, then we have the following properties: 
\begin{points}
\item if $X$ is irreducible, then $R(X)$ is irreducible;
\item $\dim R(X)=2\dim X;$
\item if $f:Y\to X$ is a closed (resp. open) immersion, then the induced morphism $R(f):R(Y)\to R(X)$ is a closed (resp. open) immersion.
\end{points}
\end{pro}

Then we get the following easy consequence.
\begin{lem}\label{lemzarcloweil}Let $Y\in Var_{\C}$ and $X$ be a closed subset $Y$. Then $R(X)$ is the Zariski closure of $X(\C)=R(X)(\R)$ in $R(Y).$
\end{lem}
\proof
We may assume that $X$ and $Y$ are irreducible. It is clear that $R(X)(\R)\subseteq R(X)$. So $\overline{R(X)(\R)}^{\zar}\subseteq R(X).$ 
Since $$\dim_{\R}\overline{R(X)(\R)}^{\zar}\geq \dim_{\R} R(X)(\R)=2\dim X=\dim R(X)$$ and $R(X)$ is irreducible, we get $\overline{R(X)(\R)}^{\zar}= R(X).$
\endproof

\medskip

We denote by $\sigma\in \Gal(\C/\R)$ the complex conjugation $z\mapsto \overline{z}$. For every complex variety $X$, one denote by $X^{\sigma}$ the base change of $X$ by the field extension $\sigma: \C\to \C$. This induces a morphism of schemes (over $\Z$) $\sigma: X^{\sigma}\to X$. It is not a morphism of schemes over $\C$.
It is clear that $(X^{\sigma})^{\sigma}=X.$
\begin{exe}
If $X$ is the subvariety of $\A^N_{\C}=\Spec \C[z_1,\dots, z_N]$ defined by the equations 
$\sum_{I}a_{i,I}z^I=0,  i=1,\dots, s$
Then $X^{\sigma}$ is the subvariety of $\A^N_{\C}$ defined by $\sum_{I}\overline{a_{i,I}}z^I=0,  i=1,\dots, s$.
The map $\sigma: X=(X^{\sigma})^{\sigma}\to X^{\sigma}$ sends a point $(z_1,\dots,z_N)\in X(\C)$ to $(\overline{z_1},\dots,\overline{z_N})\in X^{\sigma}(\C)$.
\end{exe}

The following result due to Weil is useful for computing the Weil restriction.
\begin{pro}\cite[Exercise 4.7]{Poonen2017}\label{proweilconj}
We have a canonical isomorphism $$R(X)\otimes_{\R}\C\simeq X\times X^{\sigma}.$$ Under this isomorphism,  $$R(X)(\R)=\{(z_1,z_2)\in X(\C)\times X^{\sigma}(\C)|\,\, z_2=\sigma(z_1)\}$$ and
$\tau_X$ sends $z\in X(\C)$ to $(z,\sigma(z))\in R(X)(\R).$
\end{pro}

\subsubsection{The norm map}\label{subsubsectionnormmap} For $N\geq 1$, let $\nu_N: \C^N/S_N\to \R^N/S_N$ be the real analytic map sending $\{z_1,\dots,z_N\}$ to $\{|z_1|^2,\dots, |z_N|^2\}.$
We view $\C^N/S_N$ as a real algebraic variety via the identification 
$$\C^N/S_N=(\A^N_{\C}/S_N)(\C)=R(\A^N_{\C}/S_N)(\R)\subseteq R(\A^N_{\C}/S_N)(\C).$$

The following result is the aim of this section. We postpone its proof to the end of this section.
\begin{pro}\label{propreimagenvnotz}For $a:=\{a_1,\dots,a_N\}\in \R_{>0}^N/S_N$, $\nu_N^{-1}(a)$ is real Zariski closed if and only if $N=1$ or $N=2$ and $a_1\neq a_2.$
\end{pro}

Set $X:=R(\A^N_{\C}/S_N)\otimes_\R\C=(\A^N_{\C}/S_N)\times (\A^N_{\C}/S_N).$ (Since $\A^N_{\C}/S_N$ is defined over $\R$ we have $\A^N_{\C}/S_N=(\A^N_{\C}/S_N)^\sigma$.) Consider the quotient morphisms 
$q_1: \A_{\C}^N\twoheadrightarrow \A_{\C}^N/S_N$ defined by 
$$(z_1,\dots, z_N)\mapsto \{z_1,\dots, z_N\}$$
and 
$q_2: \A_{\C}^N\times \A_{\C}^N\twoheadrightarrow X$ defined by 
$$(u_1,\dots, u_N; v_1,\dots, v_N)\mapsto (\{u_1,\dots, u_N\},\{v_1,\dots, v_N\}).$$
Consider the morphism $\mu_N: \A_{\C}^N\times \A_{\C}^N\to \A_{\C}^N$ defined by 
$$(u_1,\dots, u_N; v_1,\dots, v_N)\mapsto (u_1v_1,\dots, u_Nv_N).$$
Let $\Gamma_{\mu_N}$ be the graph of $\mu_N$ in $(\A_{\C}^N\times \A_{\C}^N)\times \A_{\C}^N$.
Set $\Gamma_N=(q_2\times q_1)(\Gamma_{\mu_N})\subseteq X\times (\A^N_{\C}/S_N)$.
Since $q_2\times q_1$ is finite, $\Gamma_N$ is an irreducible closed subvariety of  $X\times (\A^N_{\C}/S_N)$.
We view it as a correspondence between $X$ and $\A^N_{\C}/S_N$.

Let $\pi_1: X\times (\A^N_{\C}/S_N)\to X$ and $\pi_2: X\times (\A^N_{\C}/S_N)\to (\A^N_{\C}/S_N)$ be the first and the second projection.
Then $\pi_1|_{\Gamma_N}$ is a finite morphism of degree $N!.$
For every $x\in X$, the image of $x$ under $\Gamma_N$ is $\Gamma_N(x):=\pi_2(\Gamma_N\cap \pi_1^{-1}(x))$.
For a general $x\in X(\C)$, $\Gamma_N(x)$ has $N!$ points.
 Similarly, for every $y\in \A^N_{\C}/S_N$, the preimage of $y$ under $\Gamma_N$ is $\Gamma_N^{-1}(y):=\pi_1(\Gamma_N\cap \pi_2^{-1}(y)).$

\begin{lem}\label{lemirredgaminv}For every $a=\{a_1,\dots, a_N\}\in (\A^N_{\C}/S_N)(\C)$ with $a_i\neq 0, i=1,\dots, N$,
$\Gamma_N^{-1}(a)$ is irreducible and of dimension $N$.
\end{lem}
\proof
Consider the actions of $g\in S_N$ on $\A_{\C}^N\times \A_{\C}^N$ by $$g.(u_1,\dots, u_N; v_1,\dots, v_N)=(u_{g(1)},\dots, u_{g(N)}; v_{g(1)},\dots, v_{g(N)})$$
and on $\A^N_{\C}$ by $g.(z_1,\dots, z_N)=(z_{g(1)},\dots, z_{g(N)}).$
Then we have $$q_1(g.x)=q_1(x), q_2(g.x)=q_2(x).$$
Since $$\Gamma_N^{-1}(a)=q_2(\mu_N^{-1}(q_1^{-1}(\{a_1,\dots, a_N\})))$$ and 
 $$q_1^{-1}(\{a_1,\dots, a_N\})=\{g. (a_1,\dots, a_N)|\,\, g\in S_N\},$$
 we get $\Gamma_N^{-1}(a)=q_2(\mu_N^{-1}((a_1,\dots,a_N))).$
Since $\mu_N^{-1}((a_1,\dots,a_N))$ is defined by $u_iv_i=a_i, i=1,\dots, N$, it is isomorphic to 
$(\A^1\setminus \{0\})^N$, which is irreducible.
 Since $q_2$ is finite, $\Gamma_N^{-1}(a)$ is irreducible of dimension $N.$
 \endproof

\medskip

For $a=\{a_1,\dots, a_N\}\in \R_{>0}^N/S_N\subseteq (\A^N_{\C}/S_N)(\R)$, we have 
$$\Gamma_N^{-1}(a)(\R)=\Gamma_N^{-1}(a)\cap X(\R)=\cup_{g\in S_N}V_{N,g}(a)$$
where $$V_{N,g}(a)=q_2(\{(u_1,\dots, u_N;\overline{u_1},\dots, \overline{u_N})\in \C^{2N}|\,\, u_i\overline{u_{g(i)}}=a_i,1\leq i\leq N\})$$
$$=\{(\{u_1,\dots, u_N\},\{\overline{u_1},\dots, \overline{u_N}\})\in R(X)(\R)|\,\, u_i\overline{u_{g(i)}}=a_i, 1\leq i\leq N\}$$

We note that, if $g_1,g_2\in S_N$ are conjugate,  then $V_{N,g_1}(a)=V_{N,g_2}(a)$.
For every $g\in S_N$,  it can be uniquely written as a product of disjoint cycles, i.e. 
there is a partition $\{1,\dots,N\}=\sqcup_{i=1}^s I_i$
such that 
$g=\sigma_1\cdots\sigma_s$ where $\sigma_i$ acts trivially outside $I_i$ and transitively on $I_i.$
Set $$Z_{N,g}(a):=\{(u_1,\dots, u_N;\overline{u_1},\dots, \overline{u_N})\in \C^{2N}|\,\, u_i\overline{u_{g(i)}}=a_i, i=1,\dots, N\},$$ then $V_{N,g}(a)=q_2(Z_{N,g}(a)).$

\medskip

For $i=1,\dots,s$, set $m_i:=\#I_i$ and write $I_i=\{j_1,\dots, j_{m_i}\}$ with $\sigma(j_n)=j_{n+1}$, here the index $n$ is viewed in $\Z/m_i\Z.$
We define $Z_i, i=1,\dots, s$ as follows:
\begin{points}
\item[$(E_0):$] If $m_i$ is even and $\sum_{n=1}^{m_i}(-1)^n\log a_{j_n}\neq 0,$ $Z_i:=\emptyset$.
\item[$(E_1):$] If $m_i$ is even and $\sum_{n=1}^{m_i}(-1)^n\log a_{j_n}= 0,$ then  
$Z_i$ is the set of points taking forms $(U,\overline{U})\in \C^{I_i}\times \C^{I_i}$ where 
$$U=(r_{j_1}e^{i\theta},a_1r^{-1}_{j_1}e^{i\theta}, a_2a_1^{-1}r_{j_1}e^{i\theta}, \dots, a_{j_{m_i-1}}a_{j_{m_i-2}}^{-1}\dots a_1r_{j_i}^{-1}e^{i\theta})$$ for some $r_{j_1}\in \R_{>0}$ and $\theta\in \R.$
Hence $Z_i\simeq \R_{> 0}\times (\R/\Z).$
\item[$(O):$] If $m_i$ is odd, then  $Z_i$ is the set of points taking forms $(U,\overline{U})\in \C^{I_i}\times \C^{I_i}$ where 
$$U=(r_{j_1}e^{i\theta},\dots, r_{j_{m_i}}e^{i\theta}),r_{j_n}=\left(\prod_{l=0}^{m_i-1}a_{j_{n+l}}^{(-1)^l}\right)^{1/2}$$
for some $\theta\in \R.$ Hence $Z_i\simeq \R/\Z.$
\end{points}
It is easy to show that $$Z_{N,g}(a)=\prod_{i=1}^sZ_i.$$
Let $e_0(g), e_1(g)$ and $o(g)$ be the numbers of of the index $i$ that falls into the cases $(E_0), (E_1)$ and $(O)$ respectively.  
Then $Z_{N,g}(a)=\emptyset$ if $e_0(g)>0$, otherwise $$Z_{N,g}(a)\simeq \R_{>0}^{e_1(g)}\times (\R/\Z)^{e_1(g)+o(g)}.$$

\begin{lem}\label{lemvnida}We have $V_{N,\id}(a)=\nu_N^{-1}(a)$ and it is Zariski dense in $\Gamma_N^{-1}(a).$
\end{lem}
\proof
It is clear that $V_{N,\id}(a)=\nu_N^{-1}(a)$. 
By Lemma \ref{lemirredgaminv}, $\Gamma_N^{-1}(a)$ is irreducible and of dimension $N$.
Since $Z_{N,\id}(a)\simeq (\R/\Z)^N$, $V_{N,\id}(a)=q_2(Z_{N,\id}(a))$
 is of dimension $N$. Then it is Zariski dense in $\Gamma_N^{-1}(a).$
\endproof

\proof[Proof of Proposition \ref{propreimagenvnotz}]
By Lemma \ref{lemvnida}, $\nu_N^{-1}(a)=V_{N,\id}(a)$ is Zariski closed if and only if $V_{N,g}(a)\subseteq V_{N,\id}(a)$ for every $g\in S_N.$

The case $N=1$ is trivial.
If $N=2$ and $a_1\neq a_2$, then $e_0(g)>0$ for $g\in S_2\setminus \{\id\}$. Hence $V_{N,\id}(a)$ is Zariski closed.
If there is $i\neq j$ with $a_i=a_j$, let $g:=(i,j)\in S^N.$ Then $$Z_{N,g}(a)\simeq \R_{>0}\times (\R/\Z)^{N-1}$$ which is not compact.
Since $q_2$ is finite, $q_2(Z_{N,g}(a))$ is closed but not compact. Hence it does not contained in $V_{N,\id}(a)$.

Now we may assume that $N\geq3$ and $a_i\neq a_j$ for every $i\neq j$. 
We may assume that $a_1>a_2>a_3$ and $a_1=\max\{a_i, i=1,\dots, N\}.$ 
Then for every $(\{u_1,\dots, u_N\}, \{\overline{u_1},\dots, \overline{u_N}\})\in V_{N,\id}(a)$
We have $$\max\{|u_i|, i=1,\dots, N\}=a_1^{1/2}.$$

Pick $g=(1,2,3)\in S_N.$ Then $Z_{N,\id}(a)\neq\emptyset$ and
for every point
$(u_1,\dots, u_N;\overline{u_1},\dots, \overline{u_N})\in Z_{N,g}(a)$, we have
$$\max\{|u_i|, i=1,\dots, N\}\geq |u_2|=(a_1a_2a_3^{-1})^{1/2}>a_1^{1/2}.$$ 
Since $V_{N,\id}(a)=q_2(Z_{N,\id}(a))$, $V_{N,g}(a)\cap V_{N,\id}(a)=\emptyset.$ Hence $V_{N,\id}(a)$ is not Zariski closed.
\endproof

\subsubsection{The example}
In this section, we focus on the first length spectrum map $L_1: {\rm Rat}_2(\C)\to \R_{>0}^3/S_3.$
We view ${\rm Rat}_2(\C)$ as a real algebraic variety via identifying ${\rm Rat}_2(\C)$ with $R({\rm Rat}_2)(\R)$
\begin{thm}\label{thmexamplenonalg}For $a\in (1,\sqrt{2})$, $L_1^{-1}(\{a,a,a\})$ is not real algebraic in ${\rm Rat}_2(\C).$
\end{thm}

\proof
We follow the notations in Section \ref{subsubsectionnormmap}.

\medskip

\par Recall the first multiplier spectrum  map $s_1: {\rm Rat}_2(\C)\to (\A^3/S_3)(\C).$
Then $L_1^{-1}(\{a,a,a\})=s_1^{-1}(\nu_3^{-1}(\{a^2,a^2,a^2\})).$ Set $b:=\{a^2,a^2,a^2\}.$
Since $s_1$ factors through the moduli space $\sM_2(\C)$, there is a morphism $[s_1]: \sM_2(\C)\to (\A^3/S_3)(\C)$ such that
$[s_1]\circ \Psi_2=s_1.$
It was proved by Milnor\cite{Milnor1993} that $[s_1]$ is an isomorphism to its image $M$ (see also \cite[Theorem 2.4.5]{Silverman2012}).
Moreover, by \cite[Theorem 2.4.5 and Lemma 2.4.6]{Silverman2012}, $M=q_1(Y_0)$ and $R(M)=q_2(R(Y_0))$, where 
$$Y_0=\{(z_1,z_2,z_3)\in \C^3|\,\, z_1z_2z_3=z_1+z_2+z_3-2, z_1z_2\neq 1\}\cup \{(1,1,z_3)\}.$$
\par  Set $Y:=\{(z_1,z_2,z_3)\in \C^3|\,\, z_1z_2z_3=z_1+z_2+z_3-2\}$ which is the Zariski closure of $Y_0.$
 The Zariski closure of $R(M)$ in $R(\A^3_{\C}/S_3)$ is $q_2(R(Y))$.

\begin{lem}\label{lemycapgirr}The intersection $q_2(R(Y))\cap \Gamma^{-1}_3(b)$ is irreducible of dimension $1.$
\end{lem}
\proof Observe that $(q_2(R(Y))\cap \Gamma^{-1}_3(b))\otimes_{\R}\C=q_2(Z)$ where $Z$ is the closed subset of $R(\A^3_{\C})\otimes_{\R} \C=\A_{\C}^{3}\times \A_{\C}^{3}=\Spec \C[u_1,u_2,u_3,v_1,v_2,v_3]$ 
 defined by the following equations:
\begin{points}
\item $u_1u_2u_3=u_1+u_2+u_3-2;$
\item $v_1v_2v_3=v_1+v_2+v_3-2;$
\item $u_1v_1=a$;
\item $u_2v_2=a$;
\item $u_3v_3=a$.
\end{points}
Using symmetric polynomials, one may write $$R(\A_{\C}^3/S_3)\otimes_{\R}\C=\A_{\C}^3/S_3\times \A_{\C}^3/S_3$$ as $$\A_{\C}^3\times \A_{\C}^3=\Spec \C[x,y,z,x',y',z']$$
and in this coordinate, $q_2$ is given by 
 $x\mapsto u_1+u_2+u_3$, $y\mapsto u_1u_2+u_1u_3+u_2u_3$, $z\mapsto u_1u_2u_3$, $x'\mapsto v_1+v_2+v_3$, $y'\mapsto v_1v_2+v_1v_3+v_2v_3$ and $z'\mapsto v_1v_2v_3.$
One may compute that $q_2(Z)$ is defined by the following equations:
\begin{points}
\item $z\neq 0;$
\item $x=z+2;$
\item $y=(2z+a^3)/a;$
\item $x'=a^3/z+2$;
\item $y'=a^2(z+2)/z$;
\item $z'=a^3/z$.
\end{points}
Then it is irreducible of dimension $1$ since it is parametrized by a single variable  $z$.
\endproof
Then $R(M)\cap \Gamma^{-1}_3(b)$ is irreducible, and if this intersection is nonempty,  it is of dimension $1.$
We note that $$\nu_3^{-1}(b)=M\cap q_2(Z_{3,\id}(b)).$$
Let $g=(1,2)\in S_3$.
We have $$M\cap (q_2(Z_{3,\id}(b))\cup q_2(Z_{3,g}(b)))\subseteq (R(M)\cap \Gamma^{-1}_3(b))(\R).$$
\begin{lem}\label{lemleveldimone}
	Both $M\cap q_2(Z_{3,\id}(b))$ and $M\cap q_2(Z_{3,g}(b))$ are infinite and $M\cap q_2(Z_{3,g}(b))\not\subseteq M\cap q_2(Z_{3,\id}(b)).$
\end{lem}
\proof
Since $q_2$ is finite, we only need to show that $Y_0\cap Z_{3,\id}(b)$ and $Y\cap Z_{3,g}(b)$ are infinite and $M\cap q_2(Z_{3,g}(b))\not\subseteq M\cap q_2(Z_{3,\id}(b)).$

\medskip

Since $a>1$, one may compute that $Y_0\cap Z_{3,\id}(b)=Y\cap Z_{3,\id}(b)$ and it is the set of points $(u_1,u_2,u_3)\in \C^3$
satisfying the following equations:
\begin{equation}\label{equvtid}
u_1u_2u_3=u_1+u_2+u_3-2 \text{ and }
 |u_1|=|u_2|=|u_3|=a.
\end{equation}
Consider the function $F: [0,\pi]^2\to [0,+\infty)$ given by 
$$F:(\theta_1,\theta_2)\mapsto \left|\frac{a(e^{i\theta_1}+e^{i\theta_2})-2}{a^3e^{i(\theta_1+\theta_2)}-a}\right|.$$
Since $a>1$, it is well-defined and continuous. 
We have
$$F(0,0)=|(2a-2)/(a^3-a)|=\frac{2}{a(a+1)}<1$$
and 
$$F(\pi,\pi)=|(-2a-2)/(a^3-a)|=\frac{2}{a(a-1)}>1.$$
There is $\beta\in (0, \pi/2)$ such that for every $\alpha\in [0,\beta]$, we have 
$$F(0,\alpha)<1 \text{ and } F(\pi-\alpha,\pi)>1.$$
Hence for every $\alpha\in [0,\beta]$, there is $\theta(\alpha)\in [0,\pi-\alpha]$ such that 
$$F(\theta(\alpha), \theta(\alpha)+\alpha)=1.$$
One may check that $$u_1=ae^{i\theta(\alpha)}, u_2=ae^{i\theta(\alpha)+\alpha}, u_3=a\frac{a(e^{i\theta(\alpha)}+e^{i\theta(\alpha)+\alpha})-2}{a^3e^{i(2\theta(\alpha)+\alpha)}-a}, \alpha\in [0,\beta]$$
are infinitely many distinct solutions of (\ref{equvtid}).
So $Y_0\cap Z_{3,\id}(b)$ is infinite. 

\medskip

Since $a>1$, one may compute that $Y_0\cap Z_{3,g}(b)=Y\cap Z_{3,g}(b)$ and it is the set of points $(u_1,u_2,u_3)\in \C^3$
satisfying the following equations:
\begin{equation}\label{equvtg}
u_1u_2u_3=u_1+u_2+u_3-2 \text{ and }
 u_1\overline{u_2}=|u_3|^2=a^2.
\end{equation}
Consider the function $G: \R_{>0}\times [0,\pi]\to [0,+\infty)$ given by 
$$G:(r,\theta)\mapsto \left|\frac{a(r+1/r)e^{i\theta}-2}{a^3e^{2i\theta}-a}\right|.$$
Since $a>1$, it is well-defined and continuous. 
We note that $G(1,\theta)=F(\theta,\theta)$ for $\theta \in [0,\pi].$
So $G(1,0)<1$ and $G(1,\pi)>1.$
There is $R>1$ such that for every $r\in [1,R]$,
$G(r,0)<1$ and $G(r,\pi)>1.$
Then for every $r\in [1,R]$, there is $\theta_r\in [0,\pi]$ such that 
$G(r,\theta_r)=1.$

One may check that $$u_1(r)=are^{i\theta_r}, u_2(r)=ar^{-1}e^{i\theta_r}, u_3(r)=a\frac{a(r+1/r)e^{i\theta_r}-2}{a^3e^{2i\theta_r}-a}, r\in [1,R]$$
are infinitely many distinct solutions of (\ref{equvtid}).
So $Y_0\cap Z_{3,g}(b)$ is infinite. Moreover, if $r>1$, then $\max\{|u_1(r)|,|u_2(r)|, |u_3(r)|\}=ar>a$, so $\{u_1(r), u_2(r),u_3(r)\}\in (M\cap q_2(Z_{3,g}(b)))\setminus (M\cap q_2(Z_{3,\id}(b))).$ This concludes the proof.
\endproof

Since $M\cap q_2(Z_{3,\id}(b))$ is infinite and $\dim R(M)\cap \Gamma^{-1}_3(b)=1$, the Zariski closure of $M\cap q_2(Z_{3,\id}(b))$ in $R(M)$ is $R(M)\cap \Gamma^{-1}_3(b)$ but 
$M\cap q_2(Z_{3,\id}(b))\subsetneq (R(M)\cap \Gamma^{-1}_3(b))(\R).$
So $L_1^{-1}(\{a,a,a\})=s_1^{-1}(M\cap q_2(Z_{3,\id}(b)))$ is Zariski dense in $R(s_1)^{-1}(R(M)\cap \Gamma^{-1}_3(b))$, where $R(s_1): R({\rm Rat_{2}})\to R(M)$ is induced by $s_1.$
Since $M\cap q_2(Z_{3,\id}(b))\subsetneq (R(M)\cap \Gamma^{-1}_3(b))(\R)$ and $M$ is the image of $s_1$, $L_1^{-1}(\{a,a,a\})\subsetneq R(s_1)^{-1}(R(M)\cap \Gamma^{-1}_3(b)).$
This concludes the proof.
\endproof


\subsection{Images of algebraic subsets under \'etale morphisms}\label{subsectionimageetale}
Let $X$ be a variety over $\R$.
A closed subset $V$ of $X(\R)$ is called \emph{admissible} if there is a morphism 
$f: Y\to X$ of real algebraic varieties and a Zariski closed subset $V'\subseteq Y$ such that 
$V=f(V'(\R))$ and $f$ is \'etale at every point in $V'(\R).$

Every algebraic subset of $X(\R)$ is admissible.

\begin{rem}
Denote by $J$ the non-\'etale locus for $f$ in $V$. We have $J\cap V(\R)=\emptyset.$
Since we may replace $V$ by $V\setminus J$,  in the above definition we may further assume that $f$ is \'etale.
\end{rem}

\begin{rem}
Let $Y$ be a Zariski closed subset of $X$.
Since \'etale morphisms are preserved under base changes, if $V$ is admissible as a subset of $X(\R)$, it is admissible as a subset of $Y(\R).$
\end{rem}

\begin{rem}
An admissible subset is semialgebraic. So it has finitely many connected components.
\end{rem}

\begin{pro}\label{probasicadm}Let $V_1,V_2$ be two  admissible closed subsets of $X(\R)$. Then $V_1\cap V_2$ is admissible.
\end{pro}
\proof
There is a morphism 
$f_i: Y_i\to X, i=1,2$ of algebraic varieties and a Zariski closed subset $V_i'\subseteq Y_i$ such that 
$V_i=f(V_i'(\R))$ and $f_i$ is \'etale.
Then the fiber product $f: Y_1\times_XY_2\to X$ is \'etale. Since
$$V_1\cap V_2=f_1(V_1'(\R))\cap f_2(V_2'(\R))=f((V_1'\times_XV_2')(\R)),$$
$V_1\cap V_2$ is admissible.
\endproof

The key result in this section is the following, which shows that admissible subsets satisfy the descending chain condition.
\begin{thm}\label{thmNoetherianad}Let $V_n, n\geq 0$ be a sequence of decreasing admissible subsets of $X(\R)$. Then there is $N\geq 0$ such that $V_n=V_N$ for all $n\geq N.$
\end{thm}

We need the following lemma.
\begin{lem}\label{lemconnect}
Let $V$ be an  admissible closed subset of $X(\R)$. 
Assume that $X$ and $\overline{V}^{\zar}$ are smooth. Then $V$ is a finite union of connected components of $\overline{V}^{\zar}(\R).$
\end{lem}
\proof 
Since $\overline{V}^{\zar}$ is smooth, different irreducible components of $\overline{V}^{\zar}$ do not meet. So we may assume that $\overline{V}^{\zar}$ is irreducible of dimension $d.$ Hence $\overline{V}^{\zar}(\R)$ is smooth, it is of dimension $d$  everywhere.

There is a morphism 
$f: Y\to X$ of algebraic varieties and a Zariski closed subset $V'\subseteq Y$ such that 
$V=f(V'(\R))$ and $f$ is \'etale at every point in $V'(\R).$
After replacing $V'$ by $\overline{V'(\R)}^{\zar}$, we may assume that $V'(\R)$ is Zariski dense in $V'.$

For $x\in V$, there is $y\in V'(\R)$ such that $V'(\R)$ has dimension $d$ at $y$. Since $f$ is \'etale, $f^{-1}(\overline{V}^{\zar}(\R))$ is smooth and of dimension $d.$
Hence $V'$ coincides with $f^{-1}(\overline{V}^{\zar})$ in some Zariski open neighborhood of $y$. So $V'(\R)$ is smooth at $y$. It follows that  $f$ maps some Euclidean neighborhood of $y$ in $V'(\R)$ to some Euclidean neighborhood of $x$ in $\overline{V}^{\zar}(\R)$. This shows that $V$ is open in $\overline{V}^{\zar}(\R).$ Then $V$ is a finite union of connected components of $\overline{V}^{\zar}(\R).$
\endproof

\proof[Proof of Theorem \ref{thmNoetherianad}]
We do the proof by induction on $\dim X$. When $\dim X=0$, Theorem \ref{thmNoetherianad} is trivial. 

There is $N\geq 0$ such that $\overline{V_n}^{\zar}$ are the same for $n\geq N.$
After removing $V_n, n=1,\dots, N$, we may assume that $\overline{V_n}^{\zar}, n\geq 0$ are the same variety. After replacing $X$ by this variety, we may assume that 
$\overline{V_n}^{\zar}=X$ for all $n\geq 0.$ Let $X_0,X_1$ be the smooth and singular part of $X$.  We only need to show that both $V_n\cap X_0(\R), n\geq 0$ and $V_n\cap X_1(\R), n\geq 0$ are stable for $n$ large. Since $\dim X_1<\dim X$, $V_n\cap X_1(\R), n\geq 0$ is stable for $n$ large by the induction hypothesis. Since $X_0$ is smooth, by Lemma \ref{lemconnect}, every $V_n$ is a union of connected components of $X_0(\R).$ Since $X_0(\R)$ has at most finitely many connected components, we conclude the proof.
\endproof

\rem Theorem \ref{thmNoetherianad} does not hold for general semialgebraic subsets.
The following example shows that it does not hold even for images of algebraic subsets under finite morphisms.
For $n\geq 0$, set $Z_n:=[n,\infty)\subseteq \A^1(\R).$ They are the images of $\A^1(\R)$ under the finite morphisms $z\mapsto z^2+n, n\geq 0$. We have $Z_{n+1}\subset Z_n$ but $\cap_{n\geq 0} Z_n=\emptyset$. 
\endrem

Let $d\geq 2.$
We now view $\text{Rat}_d(\C)$ as a real variety and study the locus in it with given length spectrum.
For $n\geq 1$, $s=1,\dots, N_n$ and $a\in \R^s/S_s$, let $\La^s_n(a)$ be the subset of $t\in \text{Rat}_d(\C)$ such that 
$a\subseteq L_n(t)$ i.e. $f_t^n$ has a subset of fixed points counting with multiplicity,  such that the set of norms of multipliers of these fixed points equals to $a.$ It is a closed subset in $\text{Rat}_d(\C)$.

\rem
 This notion generalizes the notion $\La_n(a)$. When $s=N_n$, we get  $\La_n(a)=\La_n^{s}(a).$
\endrem

Pick $(a_1,\dots, a_s)\in \R^s$ representing $a\in [0,+\infty)^s/ S_s$, we have $$\La^s_n(a)=\phi^s_n(|\la_n^s|^{-1}(a_1,\dots,a_s)).$$

Even though $|\la^s_n|$ is not real algebraic, its square $|\la^s_n|^2$ is real algebraic. So
$|\la_n^s|^{-1}(a_1,\dots,a_s)=(|\la_n^s|^2)^{-1}(a_1^2,\dots,a_s^2)$ is real algebraic.
Hence $\La^s_n(a)$ is semialgebraic. 
Moreover, if $a_i\neq 1$ for every $i=1,\dots, s$, $$|\la_n^s|^{-1}(a_1,\dots,a_s)\subseteq (\la_n^s)^{-1}((\A^1\setminus \{1\})^s).$$
So $\phi^s_n$ is \'etale along $|\la_n^s|^{-1}(a_1,\dots,a_s).$ This shows the following fact.
\begin{pro}\label{prolenlcoadm}
For $a\in ([0,+\infty)\setminus \{1\})^s/S_s$, $\La^s_n(a)$ is admissible.
\end{pro}

\subsection{Length spectrum}
Let $f$ be an endomorphism of $\P^1(\C)$ of degree $d\geq 2$.
Recall that the \emph{length spectrum} $L(f)=\{L(f)_n, n\geq 1\}$ of $f$ is a sequence of finite multisets, where $L(f)_n:=L_n(f)$ is the multiset of 
norms of multipliers of all fixed points of $f^n.$
In particular, $L(f)$ is a multiset of positive real numbers of cardinality $d^n+1$. 
For every $n\geq 0$, let $RL(f)_n$ be the sub-multiset of $L(f)_n$ consisting of all elements $>1.$
We call $RL(f):= \{RL(f)_n, n\geq 1\}$ the \emph{repelling length spectrum} of $f$ and 
$RL^*(f):= \{RL^*(f)_n:=RL(f)_{n!}, n\geq 1\}$ the \emph{main repelling length spectrum} of $f$.  We have $d^{n}+1\geq|RL(f)_n|\geq d^{n}+1-M$ for some $M\geq 0$.
It is clear that the difference  $d^{n!}+1-|RL^*(f)_n|$ is increasing and bounded.


Let $\Om$ be the set of sequences $A_n, n\geq 0$ of multisets consisting of real numbers of norm strictly larger than $1$  satisfying $|A_n|\leq d^{n!}+1$ 
and for every $a\in A_n$ with multiplicity $m$, $a^{n+1}\in A_{n+1}$ with multiplicity at least $m$.
For $A,B\in \Om$, we write $A\subseteq B$ if $A_n\subseteq B_n$ for every $n\geq 0.$
An element $A=(A_n)\in \Om$ is called \emph{big} if $d^{n!}+1-|A_n|$ is bounded. 
For every endomorphism $f$ of $\P^1(\C)$ of degree $d$, we have $RL^*(f)\in \Om$ and it is big.

\medskip

For $A\subseteq RL^*(f)$,  by induction, we can show that there is a sequence of sub-multisets $P_n\subseteq \Fix_{n!}(f), n\geq 1$ (here we view  $\Fix_{n!}(f)$ as a multiset of cardinal $d^{n!}+1$) such that $P_n\subseteq P_{n+1}$ and $A_n=\{|df^{n!}(x)||\,\,x\in P_n\}$.
Such $P:=(P_n)$ is called a \emph{realization of $A$}, which may not be unique.  
Further assume that $A$ is big, then for every realization of $A$, $|\Fix_{n!}(f)\setminus P_n|$ is bounded. It follows that $\Per(f)\setminus (\cup_{n\geq 0}P_n)$ is finite. 
\medskip

Let $A\in \Om.$ Define $\La(A):=\cap_{n\geq 1}\La^{|A_n|}_{n!}(A_n),$ which is the locus of $t\in \text{Rat}_d$ satisfying $A\subseteq RL^*(f_t)$.
It is clear that $\La^{|A_n|}_{n!}(A_n), n\geq 1$ is decreasing, and by Proposition \ref{prolenlcoadm}, each of them is admissible. Hence by Theorem \ref{thmNoetherianad} we get the following result.
\begin{pro}\label{prolengsadm}There is $N(A)\geq 0$ such that  $$\La(A)=\La^{|A_{N(A)}|}_{N(A)!}(A_{N(A)}),$$ which is admissible. 
\end{pro}

Let $\gamma\simeq [0,1]$ be a real analytic curve in $\text{Rat}_d(\C)$, we view $\gamma\times \P^1(\C)$ as a subset of $\text{Rat}_d(\C)\times \P^1(\C)$. 
Let $f_{\gamma}$ be the restriction of $f_{\text{Rat}_d(\C)}$ to $\gamma\times \P^1(\C)$.
For every $n$-periodic point $x=(t,y)\in \gamma\times \P^1(\C)$, let $\gamma^n_x$ be the connected component of 
$$(\gamma\times \P^1(\C))\cap \text{Rat}_d(\C)[n]=\phi_n^{-1}(\gamma)$$ containing $x.$
\begin{rem}\label{remreprx}
If $x$ is repelling for $f_t$, then $\phi_n$ is \'etale at $(t,x)$, hence it induces an isomorphism from some neighborhood of $(x,t)$ in $\gamma^n_x$ to its image in $\gamma.$

Moreover, if $|\la_n|(\gamma^n_x)\subseteq  (1,+\infty)$, then $\phi_n$ is \'etale along $\gamma^n_x$, in particular $\phi_n|_{\gamma^n_x}: \gamma^n_x\to \gamma$ is a covering map. Since $\gamma$ is simply connected, $\phi_n|_{\gamma^n_x}: \gamma^n_x\to \gamma$ is an isomorphism.  
If $n| m$, then $\gamma^n_x\subseteq \gamma^m_x$. On the other hand, for every $(u,y)\in \gamma^n_x$, the multiplicity of $y$ in $\Fix(f^m_u)$ is $1.$  
So $\gamma^m_x$ coincide with $\gamma^n_x$ in a neighborhood of $y.$ Hence $\gamma^m_x=\gamma^n_x$.
This implies that every $y\in \gamma_x$ has the same minimal period and for every period $l$ of $y$, $\gamma_y^l=\gamma_x^n.$
\end{rem}
\begin{lem}\label{lemcontfamily} 
Fix $A\in \Omega$. Assume that for every $t\in \gamma$, $A\subseteq RL^*(f_t)$.
Then there is a realization $P$ of $A$ for $f_0$, such that the following holds:
\begin{points}
\item For every $x\in \cup_{n\geq 0}P_n$, $\gamma^{m}_{(0,x)}$ does not depend on the choice of period $m$ of $x.$
We denote by $\gamma_x=\gamma^{m}_{(0,x)}$ for some (then every) period $m$ of $x$.
Then $\phi_m|_{\gamma_x}: \gamma_x\to \gamma$ is a homeomorphism and it is \'etale along $\gamma_x.$ 
In particular, for different points $x$, $\gamma_x$ are disjoint.
\item For every $x\in \cup_{n\geq 0}P_n$, with a period $m$, $|\la_m|$ is a constant on $\gamma_x.$ 
\end{points}
\end{lem}

\proof
For every $n\geq 1$, let $B_n$ be the subset of $\Fix(f_0^{n})$ such that $|\la_n|$ is a constant $>1$ on $\gamma_{(0,x)}^n$.
If $x\in B_n$ for some $n\geq 1$, by Remark \ref{remreprx}, $x\in B_m$ for every period $m$ of $x$ and 
$\gamma_x:=\gamma^m_{(0,x)}$ does not depend on the choice of period $m.$ Moreover, $\phi_m|_{\gamma_x}: \gamma_x\to \gamma$ is a homeomorphism and it is \'etale along $\gamma_x.$ In particular for for different points $x$, $\gamma_x$ are disjoint.

It is clear that $B=(B_{n!})$ realizes an element $C\in \Om.$ We only need to show that $A\subseteq C.$
Let $a$ be an element in $A_n$ of multiplicity $l\geq 1.$
Then for every $t\in \gamma$,  since $|a|>1$, $|\la_{n!}|^{-1}(a)\cap \phi_{n!}^{-1}(t)$ contains at least $l$ distinct points. 
Let $x_1,\dots, x_s$ be the elements in $x\in B_{n!}$ with $\la_{n!}((0,x))=a.$ We only need to show that $s\geq l.$ 
For every $i=1,\dots,s$, $\gamma_{x_i}$ is a connected component of $\phi_{n!}^{-1}(\gamma).$
Set $Z:=\phi_{n!}^{-1}(\gamma)\setminus \cup_{i=1}^s\gamma_{x_i}.$ 
If $s<l$, then for every $t\in \gamma$, $Z\cap |\la_{n!}|^{-1}(a)\cap \phi_{n!}^{-1}(t)$ has at least one point. So there is $y\in Z$ such that 
$\gamma^{n!}_z\cap |\la_{n!}|^{-1}(a)$ is infinite. Since both $\gamma^{n!}_z$ and $|\la_{n!}|^{-1}(a)$ are real analytic in $\gamma\times \P^1(\C)$,
$\gamma^{n!}_z\subseteq |\la_{n!}|^{-1}(a).$
By Remark \ref{remreprx}, $\gamma^{n!}_z$ meets $\phi_{n!}^{-1}(0)$ at some point $(0,x)$ for some $x\in B_n.$ So $\gamma^{n!}_z=\gamma_x$, which is a contradiction. 
\endproof

\subsection{Length spectrum as moduli}
Let $\Psi: \text{Rat}_d(\C)\to \sM_d(\C)=\text{Rat}_d(\C)/\PGL_2(\C)$ be the quotient map. 
Let $FL_d(\C)\subseteq \text{Rat}_d(\C)$ be the locus of Latt\`es maps, which is Zariski closed in $\text{Rat}_d(\C)$.
We now prove Theorem \ref{thmlength} via proving the following stronger statement.
\begin{thm}\label{thmbigspecrig}
If $A\in \Om$ is big, 
then $\Phi(\La(A)\setminus FL_d(\C))\subseteq  \sM_d$ is finite.
\end{thm}

\proof
By Proposition \ref{prolengsadm}, $\La(A)$ is admissible in $\text{Rat}_d(\C)$.
Hence $\La(A)\setminus FL_d(\C)$ is admissible in $\text{Rat}_d(\C)\setminus FL_d(\C)$. 
In particular, $\La(A)\setminus FL_d(\C)$ and $\Phi(\La(A)\setminus FL_d(\C))$ are semialgebraic.

To get a contradiction, assume that $\Phi(\La(A)\setminus FL_d(\C))$ is not finite.
By Nash Curve Selection Lemma \cite[Proposition 8.1.13]{Bochnak1998}, there is a real analytic curve $\gamma\simeq [0,1]$ in $\La(A)\setminus FL_d(\C)$ whose image in $\sM_d$ is not a point. 
Since non-flexible Latt\`es exceptional endomorphisms  are isolated in the moduli space $\mathcal{M}_d$,  there is at least one $f_t$ that is not exceptional. Without loss of generality we assume $f_0$ is not exceptional. We now apply Lemma \ref{lemcontfamily} for $\gamma$ and $A$,  and follows the notation there. Set $Q:=\cup_{n\geq 0}P_n$. Then
 $S:=\Per(f_0)\setminus Q$ is finite. 

Pick any $z_0\in Q.$ By the discussion in Example \ref{horseshoe}, there exist a horseshoe $K$ of $f_0$ containing $z_0$ and $K\cap S=\emptyset$.  There is  $m\geq 0$ such that $f_0^m(K)=K$ and $f_0^m(z_0)=z_0$. By Lemma \ref{motion}, there exits $\ep>0$ and a continuous map $h:[0,\ep]\times K\to \P^1(\C)$ such that for each $t\in [0,\ep]$:
\begin{points}
	\item $K_t:=h(t,K)$ is an expanding set of $f_t^m$.
	\item the map $h_t:=h(t, \cdot):K\to K_t$ is a homeomorphism and $f_t^m\circ h_t=h_t\circ f_0^m$.
\end{points}

For every $t\in [0,\ep]$ and for every $w_0\in K$ satisfying $f_0^{nm}(w_0)=w_0$, we have $f_t^{nm}(h_t(w_0))=h_t(w_0)$. 
It follows that $h_t(w_0)=\gamma_{w_0}(t).$ Since $|\la_{nm}|$ is a constant on $\gamma_{w_0}$, we get $|df_0^{nm}(w_0)|=|df_t^{mn}(h_t(w_0))|$. We claim that $K_t$ is a CER of $f_t$.  We check $(f_t,K_t)$ satisfies Definition \ref{defcer}: since $K_t$ is expanding by Lemma \ref{motion}, (ii) holds; 
since topological exactness and openness preserved by topological conjugacy, by Remark \ref{openmap}), (i) and (iii) hold.

Since $f_0$ is not exceptional, by Theorem \ref{linearcer}, $K$ is a non-linear CER for $f_0$.  By Theorem \ref{cerrigidity}, for every fixed $t\in [0,\ep]$, the conjugacy $h_t$ can be extended to a conformal map $h_t:U\to V$ where $U$ is a neighborhood of $K$ and $V$ is a neighborhood of $K_t$.   This implies  $df_0^m(z_0)=df_t^m(\gamma_{z_0}(t))(=df_t^m(h_t(z_0)))$ or $df_0^m(z_0)=\overline{df_t^m(\gamma_{z_0}(t))}$. Since $df_t^m(\gamma_{z_0}(t))$ depends continuously on $t$, we must have  $df_0^m(z_0)=df_t^m(\gamma_{z_0}(t))$ when $t\in [0,\ep]$. Since $\gamma_{z_0}$ is real analytic, the map $t\mapsto df_t^m(\gamma_{z_0}(t))$ is real analytic on $\gamma=[0,1]$.  It is a constant on $[0,\ep]$, hence it is a constant on $\gamma.$
Let $n$ be any period of $z_0$, the above argument shows that $(\la_n|_{\gamma_{z_0}})^m$ is a constant, hence $\la_n|_{\gamma_{z_0}}$ is a constant.
\medskip

Since our choice of $z_0\in Q$ is arbitrary, for every $z_0\in Q$, of period $n$, the map $t\mapsto df_t^n(\phi(t))$ is a constant on $[0,1]$.  Since $S$ is finite, all $f_t$ have the same multiplier spectrum for periodic points of sufficiently high period. 
\medskip

The set of all endomorphisms in $\text{Rat}_d(\C)$ with the same multiplier spectrum of $f_0$ for periodic points with period at least $N\geq 1$ is an algebraic variety. We denote it by $V_N$.  there exist $N\geq 1$ such that $\gamma\subseteq V_N.$ Further more there exist an irreducible component $X$ of $V_N$ which contains $\gamma.$  The irreducible variety $X$ forms a stable family (see \cite[Chapter 4]{mcmullen2016complex}), since the period of attracting cycles are bounded in $V_N$. 
The variety $X$  is not isotrivial since $\Psi(\gamma)$ is not a point. By Theorem \ref{thmmcmullen}, $\gamma\subseteq X$ is contained in the flexible Latt\`es family, which is a contradiction. 
\endproof
\medskip
\subsection*{Conflicts of interest} None.
\medskip
\subsection*{Financial support}  The second named author  is supported by the project ``Fatou'' ANR-17-CE40-0002-01.

\end{document}